\documentclass[reqno]{amsart}
\usepackage[margin=3cm]{geometry}
\usepackage{a4wide}
\usepackage{amsmath,amsfonts,amssymb,amsthm,version}
\usepackage{mathrsfs}
\usepackage{xcolor}
\usepackage{epsfig}
\usepackage{subfigure}
\usepackage{enumerate}
\usepackage{graphicx}
\usepackage{subfigure}
\usepackage{rotating}
\usepackage{stmaryrd}
\usepackage{float}
\usepackage{upgreek}
\usepackage{amsopn}
\usepackage{upgreek}
\usepackage{multicol}
\usepackage[colorlinks,
            linkcolor=red,
            anchorcolor=green,
            citecolor=blue
            ]{hyperref}
\usepackage[all]{xy}
\numberwithin{equation}{section}
\allowdisplaybreaks[1]

\providecommand{\U}[1]{\protect\rule{.1in}{.1in}}
\newtheorem{theorem}{Theorem}

\newtheorem{lemma}[theorem]{Lemma}

\newtheorem{proposition}[theorem]{Proposition}
\newtheorem{remark}[theorem]{Remark}

\begin{document}

\title[Differentiable normal linearization of partially hyperbolic DS]
{Differentiable normal linearization  of
partially
hyperbolic dynamical systems}

\author[ Weijie Lu, Yonghui Xia, Weinian Zhang, Wenmeng Zhang]
{Weijie Lu,  Yonghui Xia, Weinian Zhang, Wenmeng Zhang}  % in alphabetical order

\address{Weijie Lu, School of Mathematics  Science, Zhejiang Normal University, Jinhua  321004, China}
\email{luwj@zjnu.edu.cn}

\address{Yonghui Xia, School of Mathematics, Foshan University, Foshan 528000, China}
\email{xiadoc@163.com}

\address{Weinian Zhang,  School of Mathematics, Sichuan University,  Chengdu 610064, China}
\email{matzwn@126.com}

\address{Wenmeng Zhang, School of Mathematical Sciences, Chongqing Normal University, Chongqing 401331,  China}
\email{wmzhang@cqnu.edu.cn.  Author for correspondence}

\date{}

\begin{abstract}
A result on $C^0$ linearization which is differentiable at the hyperbolic fixed point is known. 
In this paper, we further investigate a partially hyperbolic diffeomorphism $F$ to find a local $C^0$ conjugacy, which is $C^1$ on the center manifold, to linearize the hyperbolic component (normal to the center direction) and obtain its Takens' normal form. Our result is optimal, as it needs no non-resonant condition usually required for smooth conjugacy (e.g., as in the Takens' theorem) and the $C^{1,\alpha}$ $(\alpha>0)$ smoothness condition is sharp. For the proof, the center direction obstructs the decoupling of $F$ as the stable and unstable foliations do not intersect. We overcome this difficulty via a semi-decoupling method only with the unstable foliation, where a modified Lyapunov-Perron equation needs to be established along the center direction. Subsequent issues of cocycle reduction and differentiable linearization for an expansive fiber-preserving mapping are then addressed by the Whitney's extension theory and a lifting technique, respectively.
In the local context, our result improves the result of $C^0$ normal linearization by [C. Pugh and M. Shub, Invent. Math., 10 (1970): 187-198] to a differentiable one.
\end{abstract}

\maketitle

%\tableofcontents

\section{Introduction}

Linearization is one of the most important theories in the field of dynamical systems.
Given a smooth (analytic) diffeomorphism $F$ near its fixed point $0$,
we expect to find a smooth (analytic) transformation such that $F$ can be conjugated to its linear part $\Lambda:=DF(0)$.
In 1890,  Poincar\'{e} (\cite{Poincare}) first studied the analytic linearization   of an analytic diffeomorphism $F(x) = \Lambda x + O(\|x\|^2)$ on $\mathbb{C}^d$.
He proved that if all eigenvalues $\lambda_1,...,\lambda_p$ of $\Lambda$ lie inside  (or outside) the unit circle and satisfy the non-resonant conditions of all orders, i.e.,
$$
\lambda_i\ne \lambda_1^{n_1}\cdots \lambda_d^{n_d}, \quad \forall i\in \{1,...,d\},
$$
for all $n_1,...,n_d\in \mathbb{N}\cup\{0\}$ with $\sum_{i=1}^d n_i \ge 2$, then there is an analytic diffeomorphism
$\Phi(x)$ such that $\Phi\circ F \circ \Phi^{-1}=\Lambda$.
Later, Siegel (\cite{Siegel}) considered the analytic linearization
under an assumption that the eigenvalues satisfy the Diophantine condition
$$
|\lambda_i - \lambda_1^{n_1}\cdots \lambda_d^{n_d}|\ge \frac{C}{|n|^\nu}, \quad \forall i\in \{1,...,d\},\,  \sum_{i=1}^d n_i \ge 2,
$$
for some constants $C, \nu>0$, proving that there exists an analytic diffeomorphism $\Phi(x)$ linearizing $F$.
Brjuno (\cite{Brj-7172}) further weakened Siegel's condition and proved the same result.
The proofs of Poincar\'{e}'s
and Siegel’s analytic linearization theorems, respectively,
can also be consulted in the works of Arnold (\cite{AV-book}),
%Meyer (\cite{Mey-75}),
Moser (\cite{Mos-66}), Zehnder (\cite{Zeh-77}) and others.

In 1960s, Hartman (\cite{Hartman}) and Grobman  (\cite{Grobman})  independently proved a $C^0$ linearization of $C^1$ hyperbolic diffeomorphisms near the fixed point.
The Hartman-Grobman theorem states that, for a \(C^1\) hyperbolic diffeomorphism \(F\) on $\mathbb{R}^d$,
there exists a homeomorphism \(\Phi\) such that $F$ can be $C^0$ linearized.
Efforts were also made to generalize this result from a vicinity of a fixed point to a vicinity of an invariant manifold.
In 1970, Pugh and Shub  (\cite{PS-InventMath})  provided a normal linearization along a compact $C^1$ invariant submanifold for a normally hyperbolic diffeomorphism of a Riemannian manifold.
Compared to the Hartman-Grobman theorem, Pugh-Shub's $C^0$ normal  linearization theorem not only involves more techniques such as the fiber bundle theory and the $\lambda$-lemma, but also has more applications to problems such as the singular perturbation theory (\cite{Fen-JDE}).

Regarding smooth linearization on $\mathbb{R}^d$,
Sternberg (\cite{Sternberg1,Sternberg2}) obtained a local $C^r$ linearization of $C^N$ hyperbolic diffeomorphisms under the $N$-th order non-resonant condition, where $N = N(r)$ is a  large integer.
There is a a large number of literature to find better conditions for smooth linearization near the fixed point.
Some classical results can be found in Belicki\u{i} (\cite{Bel-78}),
ElBialy (\cite{ElB-01}),
 Hartman (\cite{Har-60}),
 %Chen (\cite{Chen-63}),
%Banyaga, de la Llave and Wayne (\cite{BDW-96}),
   Li and Lu (\cite{LL-CPAM}),
Rodrigues and Sol\`{a}-Morales (\cite{RS-JDDE}),
  Zhang, Zhang and Jarczyk (\cite{ZZJ-MA}).
 % Cuong, Doan and Siegmund (\cite{CDS-JDDE}),
 %Tong, Xu and Li (\cite{TXL-JDE}).

On the other hand, extending Sternberg's theorem, Takens (\cite{T-Top}) proved that under the $N$-th order strong non-resonant condition, i.e.,
$$
|\lambda_i|\ne |\lambda_1|^{n_1}\cdots |\lambda_d|^{n_d}, \quad \forall i\in \{1,...,d\},
$$
with $2\le \sum_{i=1}^d n_i\le N$,
the hyperbolic part (i.e., the normal direction) of a $C^N$ partially hyperbolic diffeomorphism $F$ near its center manifold can be $C^k$ linearized to obtain its Takens' normal form (see (\ref{Co-H}) below).
For recent progresses in this regard, one can refer to  (\cite{DTZ-SCM,LL-DCDSB16}).

Notice that among the results mentioned above, $C^0$ linearization does not require the non-resonant condition, while analytic or smooth linearization needs it. Although some known results (e.g., \cite{Ray-JDE,TB-JDE}) tell that $C^0$ linearization can be improved to H\"older linearization without the non-resonant condition, the H\"older linearization cannot distinguish important properties such as the characteristic directions of dynamical systems.

Thus, one expects to enhance the regularity of the Hartman-Grobman  $C^0$ linearization as much as possible without the non-resonant condition. However, Hartman's counterexample (\cite{Har-60}) shows that even a $C^\infty$ diffeomorphism does not admit $C^1$ linearization with resonance. Therefore, efforts are made to prove that the conjugacy is differentiable at the fixed point, i.e, the conjugacy $\Phi$ has the form
$$
\Phi(x)=x+o(\|x\|)\quad {\rm as}~x\to 0,
$$
due to its important applications to the studies of semilinear elliptic equations (\cite[pp. 342]{GGS-BOOK}), nonhyperbolic singularities (\cite{BDS-ETDS}), coarse-grained entropy (\cite{PT-10}), Onsager-Machlup functional (\cite{DLLR-SCM}) and so on.

The result on differentiable linearization without any non-resonant condition can be first found in \cite{vS-JDE} by van Strein in 1990 and later in \cite{Ray-JDE} and \cite{GHR-DCDS},
which stated that a smooth hyperbolic diffeomorphism on $\mathbb{R}^d$ admits a local $C^0$ linearization which is differentiable at the fixed point.
Recently, Zhang, Lu and Zhang (\cite{ZLZ-TAMS}) considered a $C^{1,\alpha}$ ($\alpha>0$) diffeomorphism $F$ and employed the method of decoupling via stable and unstable foliations to prove the differentiable linearization in a
Banach space under the spectral bandwidth condition
\begin{align*}%\label{S-ban}
\lambda_i^+/\lambda_i^- < (\lambda_k^+)^{-\alpha}, \, \forall i=1,...,k, \quad \lambda_j^+/\lambda_j^- < (\lambda_{k+1}^-)^{\alpha}, \, \forall j=k+1,...,d,
\end{align*}
where the spectrum of $\Lambda$ satisfy
$$
\lambda_1^-\le \lambda_1^+<\cdots <\lambda_k^- \le \lambda_k^+ <1 < \lambda_{k+1}^- \le \lambda_{k+1}^+ < \cdots < \lambda_d^- \le \lambda_d^+.
$$
Since the spectral condition can be automatically satisfied in $\mathbb{R}^d$, this result generalizes the previous results of van Strein et al. by lowering the smoothness condition to $C^{1,\alpha}$, which is sharp via a counterexample (i.e., a $C^1$ diffeomorphism) given in \cite{ZLZ-TAMS}. Moreover, the conjugacy was further proved to have the form
$$
\Phi(x)=x+O(\|x\|^{1+\beta})\quad {\rm as}~x\to 0,
$$
with $\beta>0$. This result was later applied to study the Onsager-Machlup functional (\cite{DLLR-SCM}),
so that the derivative of its minimizing sequence is uniformly bounded variation in the neighborhood of the critical point, which relies on the key estimate $O(\|x\|^{1+\beta})$ of the nonlinearity of $\Phi$.
%thereby guaranteeing the local convergence of the minimizing sequence (see \cite[Proposition 4]{DLLR-SCM})  and facilitating the analysis of the graph limit of the  Onsager-Machlup functional over long times.

Recently,
Lu, Xia and Zhang (\cite{LXZ-pre}) presented a differentiable linearization result in a Banach space
under a weaker spectral bandwidth condition
$$
\lambda_i^+/\lambda_i^- < (\lambda_i^+)^{-\alpha}, \, \forall i=1,...,k, \quad \lambda_j^+/\lambda_j^- < (\lambda_{j}^-)^{\alpha}, \, \forall j=k+1,...,d,
$$
which is shown to be (almost) sharp by a counter example given in \cite{RS-Prob}.
Moreover, Dragi\v{c}evi\'{c}, Zhang and Zhang (\cite{DZZ-PLMS}) proved that the differentiable linearization result holds in the presence of nonuniform hyperbolicity.

In this paper, we  establish a differentiable normal linearzation, without any non-resonant condition, under the framework of partial hyperbolicity, to obtain the Takens' normal form (see (\ref{Co-H}) below).
More specifically, we have the following result:
\begin{itemize}
\item A 
$C^{1,\alpha}$ ($\alpha\in (0,1]$)
partially hyperbolic diffeomorphism admits a $C^0$ normal linearization, where the $C^0$ conjugacy is differentiable on the center manifold with a continuous derivative.
\end{itemize}
The conjugacy exhibits $C^1$-smoothness on the center manifold, which usually requires some non-resonant conditions (for example, as in the Takens' theorem), but we do not need.
Moreover, as mentioned before, the $C^{1,\alpha}$ smoothness condition is sharp via a counterexample given in \cite{ZLZ-TAMS}, which will be elaborated in the next section (see Remark 4).
Since the center manifold can be regarded as a local version of the normally hyperbolic invariant manifold, in the local context, our result improves the result of Pugh and Shub's $C^0$ normal linearization in \cite{PS-InventMath} to a differentiable one.

In comparison with the hyperbolic case, the center direction brings a crucial difficulty. Precisely, in the hyperbolic case, stable and unstable foliations intersect transversally, allowing to decouple the system into a contraction and an expansion; however, in the partially hyperbolic case,  the center direction prevents intersection of foliations. In order to overcome the difficulty, inspired by \cite{PS-InventMath}, we introduce the method of semi-decoupling. Namely, we straighten up only the unstable foliation to obtain an expansive fiber-preserving mapping and utilize the differentiable linearization result in \cite{ZLZ-TAMS} by a lifting technique. In the construction of the unstable foliation, in order to show its regularity on the center manifold, a modified Lyapunov-Perron equation needs to be established along the center direction.

For the subsequent difficulty in cocycle reduction posed by the semi-decoupling, we will address it by means of Whitney's extension theory.
More precisely, for an expansive fiber-preserving mapping with base points on the center-stable subspace, its linear part is actually a linear cocycle depending on the center-stable variables. This linear cocycle needs to be reduced to a simpler one depending only on the center variable becasue the linear part in the Takens' normal form (see (\ref{Co-H}) below) depends only on the center variable. For this purpose, we first establish a $\beta$-H\"{o}lder conjugacy with $\beta>0$ between these two classes of cocycles, and then find a $C^{1,\beta}$ transformation whose derivative at the center-stable subspace is just the
$\beta$-H\"{o}lder conjugacy by using the Whitney extension theorem. Such a $C^{1,\beta}$ transformation helps us achieve the reduction.

This paper will be organized as follows.
%Section \ref{Sec-ma}  concerns the main theorem and some preliminaries.
In Section \ref{Sec-ma}, we first state the main theorem, and then provide the block-diagonalization of linear cocycles generated by $DF$ on the center manifold as preliminaries.
Section \ref{Sec-4} is devoted to construct the unstable foliation and show its regularity on the center manifold. Subsequently, using the semi-decoupling method, we prove the main theorem in Section \ref{sec-diff} by utilizing the reduction of a linear cocycle and the differentiable linearization of an expansive fiber-preserving mapping, which will be complementarily proved in Sections \ref{Sec-6} and \ref{Sec-7}, respectively.

%%%%%%%%%%%%%%%%%%%%%%%%%%%%%%%
\section{Main result and preliminaries}\label{Sec-ma}

\subsection{Main result}
In this subsection, we present our main result.
Consider a
$C^{1,\alpha}$ ($\alpha\in (0,1]$) 
diffeomorphism
$F:U\to \mathbb{R}^d$ ($d\in \mathbb{N}$) defined by
\begin{equation}\label{NL-F}
F(x):=Ax+f(x), \quad \forall x\in U,
\end{equation}
where $U\subset \mathbb{R}^d$ is a small neighborhood of the origin $0$, $A\in GL(\mathbb{R},d)$, $f(0)=0$ and $Df(0)=0$. Since $U$ is small enough, we see that $f$ can be regarded as a $C^{1,\alpha/2}$ diffeomorphism such that
\begin{align*}
 \|Df(x)\| \le \delta_f, \quad 
 \|D  f(x)-D  f(y)\| \le \delta_f \|x-y\|^{\alpha/2},\quad \forall x,y\in U,
\end{align*}
where $\delta_f>0$ is a small constant depending on $U$.
Then, by using a smooth cut-off function which is equal to $1$ in $U$ and equal to $0$ outside another small neighborhood $V\subset \mathbb{R}^d$ (containing $U$) of $0$, we can extend $F$ to be a global diffeomorphism such that
\begin{align}\label{Df-hold-1}
 \|Df(x)\| \le \delta_f, \quad   
 \|D  f(x)-D  f(y)\| \le \delta_f \|x-y\|^\alpha,\quad \forall x,y\in\mathbb{R}^d.
\end{align}
Here, we still use $\delta_f$ and $\alpha$ instead of $C\delta_f$ (with a constant $C$) and $\alpha/2$, without loss of generality, as all of them can be arbitrarily small.

Next, assume
that $A$ is diagonalized in block, i.e.,
$$
A:=\mathrm{diag}\, (A_s,A_c,A_u),
$$
where
$A_s:=\mathrm{diag}\, (A_1,...,A_k)$ and  $A_u:=\mathrm{diag}\, (A_{k+1},..., A_p)$ for all $1\le k <p<d$.
Moreover, assume that all eigenvalues of $A_c$ have modulus $1$ and that all eigenvalues of $A_i$ ($i=1,...,p$) have modulus $\lambda_i$, satisfying
\begin{align}\label{sp-A}
0<\lambda_s^-<\lambda_1<\cdots<\lambda_k<\lambda_s^+<1<\lambda_u^-<\lambda_{k+1}<\cdots<\lambda_p<\lambda_u^+,
\end{align}
where $\lambda_s^-,\lambda_s^+,\lambda_u^-,\lambda_u^+$ are constants.
The above block-diagonal form of $A$ corresponds to the decomposition
\begin{align}\label{ZH}
\mathbb{R}^d=X_1 \oplus \cdots \oplus X_k \oplus X_c \oplus X_{k+1}\oplus \cdots \oplus X_p.
\end{align}
In what follows, for $x_i\in X_i$ ($i=1,...,p$),  we denote
\begin{align*}
& X_s:=X_1\oplus \cdots \oplus X_k, \quad X_u:=X_{k+1}\oplus \cdots \oplus X_p,
\\
&X_{su}:=X_s\oplus X_u,\quad X_{cs}:=X_c\oplus X_s,\quad X_{cu}:=X_c\oplus X_u,
\\
& x_s:=x_1+\cdots +x_k\in X_s, \quad x_u:=x_{k+1}+\cdots +x_p\in X_u,
\end{align*}
and, for $x=x_s+x_c+x_u\in X_s\oplus X_c\oplus X_u$, we denote
\begin{align*}
&x_{su}:=x_s+x_u,\quad x_{cs}:=x_c+x_s,\quad x_{cu}:=x_c+x_u,
\\
&\pi_i x :=x_i, \quad \pi_s x :=x_s, \quad \pi_c x:=x_c, \quad \pi_u x:=x_u,
\\
&\pi_{su}:=\pi_s+\pi_u,\quad \pi_{cs}:=\pi_c+\pi_s,\quad \pi_{cu}:=\pi_c+\pi_u.
\end{align*}
For the sake of convenience in expression, we sometimes use the Cartesian product instead of the direct sum. For example, we can rewrite (\ref{ZH}) as
$
\mathbb{R}^d=X_1 \times \cdots \times X_k \times X_c \times X_{k+1}\times \cdots \times X_p
$
since they are isomorphic.

Then we have the following result on the center manifold.

\begin{lemma}\label{cen-lem}
{\rm (\cite{Ga-CMP})}
The mapping $F$ given in \eqref{NL-F} such that \eqref{Df-hold-1} and \eqref{sp-A} hold has 
a $C^{1,\alpha}$ center manifold   
$$
\mathcal{M}_c:=\{x_c+\varpi_c(x_c):x_c\in X_c\},
$$
where $\varpi_c:X_c\to X_{su}$ is 
a $C^{1,\alpha}$ mapping 
such that
$\varpi_c(0)=0$ and $D\varpi_c(0)=0$.
\end{lemma}
\noindent By Lemma \ref{cen-lem}, we can use 
a $C^{1,\alpha}$ transformation 
$(\tilde{x}_c,\tilde{x}_{su})\mapsto(x_c,x_{su}-\varpi_c(x_c))$ to send $F$ to a new diffeomorphism whose center manifold is just the subspace $X_c$. 
We still denote the new diffeomorphism
by $F$ and assume that \eqref{Df-hold-1} holds by using the smooth cut-off function again. Note that $X_c$ is a center manifold of $F$ now.

We are in the position to state our main result of this paper.

\begin{theorem}\label{Mai-thm}
Let $F$ be 
a $C^{1,\alpha}$  diffeomorphism
on $\mathbb{R}^d$ satisfying \eqref{Df-hold-1} and \eqref{sp-A}.
Then $F$ is locally $C^0$ conjugated to its Takens’ normal form
\begin{equation}\label{Co-H}
\left(
  \begin{array}{c}
    x_s \\
    x_c \\
    x_u \\
  \end{array}
\right)
\mapsto
\left(
  \begin{array}{c}
    A_s(x_c)x_s \\
    A_c x_c+f_c(x_c) \\
    A_u(x_c)x_u \\
  \end{array}
\right)
\end{equation}
via a conjugacy $\mathcal{H}:U\subset \mathbb{R}^d\to \mathbb{R}^d$,
where $A_*(x_c):X_*\to X_*$ {\rm (}$*:=s,u${\rm )} are linear operators, being $\alpha$-H\"older with respect to $x_c\in X_c$, such that
$A_*(0)=A_*$, and $f_c:X_c\to X_c$ is 
a $C^{1,\alpha}$ mapping.
Furthermore, for any $\tilde{x}_c\in U\cap X_c$, the $C^0$ conjugacy $\mathcal{H}$ satisfies that $\mathcal{H}^{\pm 1}(\tilde{x}_c)=\tilde{x}_c$ and
\begin{align}\label{HH}
\mathcal{H}^{\pm 1}(x)=\tilde x_c+ \Delta(\tilde{x}_c)^{\pm 1}(x-\tilde{x}_c) +O(\|x-\tilde{x}_c\|^{1+\beta}) \quad {\rm as}~x\to \tilde x_c
\end{align}
for a $\beta\in (0,\alpha)$, where
$\Delta(\tilde{x}_c):\mathbb{R}^d\to \mathbb{R}^d$ is an invertible bounded linear operator depending continuously on $\tilde{x}_c$ such that $\Delta(0)=id$.
\end{theorem}

\begin{remark}{\rm From (\ref{HH}) we see that $\mathcal{H}$ is differentiable at $\tilde x_c\in {\mathcal M}_c$ such that
$D\mathcal{H}^{\pm 1}(\tilde{x}_c)=\Delta(\tilde{x}_c)^{\pm 1}$, which are continuous with respect to $\tilde{x}_c$.
It implies that the conjugacy is $C^0$ in $\mathbb{R}^d$
which exhibits $C^1$-smoothness on the center manifold. Namely, Theorem \ref{Mai-thm} shows that we  achieve the normal linearization which is differentiable
on the center manifold for a partially hyperbolic system. Remind that, in comparison with Takens' theorem, our result does not need any non-resonant condition.
}
\end{remark}

\begin{remark}
{\rm When the center direction does not exist (i.e., $\tilde{x}_c\equiv 0$), we see that \eqref{Co-H} becomes $(x_s,x_u)\mapsto (A_sx_s,A_ux_u)$ and that $\Delta(0)=id$, implying
$$
\mathcal{H}^{\pm 1}(x)=x+O(\|x\|^{1+\beta})\quad {\rm as}~x\to 0.
$$
This result is consistent with \cite[Theorem 7.1]{ZLZ-TAMS} in the hyperbolic case.
On the other hand, it is impossible to weaken the smoothness condition of $C^{1,\alpha}$ to $C^1$ in general.
In fact, even in the hyperbolic case, we can define a mapping
$$
F_{*}(x):=
\begin{cases}
\mu \int_0^x (1-\frac{1}{\log \tau}) \, d\tau, \quad &x\ne 0,
\\
0, &x=0,
\end{cases}
$$
near its fixed point $0$ with $\mu \in (0,1)$. Then, it is clear that  $F_{*}$ is a $C^1$ contraction with $DF_{*}(0)=\mu$, but it is not $C^{1,\alpha}$ near $0$ for any $\alpha>0$. As shown in \cite[p.4998-4999]{ZLZ-TAMS}, $F_{*}$
does not admit differentiable linearization at $0$, which means the sharpness of the $C^{1,\alpha}$ smoothness condition.
}
\end{remark}

%%%%%%%%%%%%%%%%%%%%%%%%%%%%%%%%%%%%%%%%%%%%%%%%%%%%%%%%%%%%%%%%%%%%%%
\subsection{Some lemmas}\label{sec-sm}

In this subsection, we consider linear cocycles generated by $DF$ on the center manifold $X_c$, and show that there exist transformations which can block-diagonalize the linear cocycles.
Due to the importance of the center variable $x_c$, in what follows, we sometimes use the notations
$$
(x_c,x_s,x_u):=x_s+x_c+x_u\in \mathbb{R}^d.  %,\quad (x_c,x_s):=x_c+x_s\in X_{cs},\quad (x_c,x_u):=x_c+x_u\in X_{cu}.
$$
Recall that the center manifold of $F$ is straightened up, i.e.,
$\pi_{su}F(x_c)=0$,
which implies that
\begin{align}\label{c-su}
\partial_{x_{c}}(\pi_{su}F)(x_c)=0,\quad \forall x_c\in X_c,
\end{align}
where we used the fact that $\partial_x h(x,0)=0$ whenever $h(x,0)=0$ for any $C^1$ mapping $h(x,y)$. Then, we give the following lemma.
\begin{lemma}\label{lm-diagC}
There is a $C^{1,\alpha}$ {\rm (}$\alpha>0$ is small{\rm )} transformation $\Upsilon:\mathbb{R}^d\to \mathbb{R}^d$ with $D\Upsilon(0)=id$ to send $F$ to a new diffeomorphism $\Upsilon\circ F\circ \Upsilon^{-1}$ {\rm (}still denoted by $F${\rm )} such that 
\begin{align}\label{CC0}
\partial_{x_{s}}(\pi_c F)(x_c)=0, \quad \partial_{x_u}(\pi_c F)(x_c)=0, \quad \forall x_c\in X_c.
\end{align}
Meanwhile, the new $F$ still satisfies \eqref{Df-hold-1} and \eqref{c-su}.
\end{lemma}
\noindent The proof of this lemma will be mainly divided into two parts: 1. Use the classical Lyapunov-Perron method to show that the tangent spaces of leaves of the stable and unstable foliations at any $x_c\in X_c$ are small perturbations of $X_s$ and $X_u$, respectively; 2. use the Whitney's extension theorem to find a $C^{1,\alpha}$ transformation to straighten up those tangent spaces
along $X_c$. Since the idea of the first part is elementary and the idea of the second part will be given in details in Section \ref{Sec-6} for the reduction of cocycles, we postpone the proof of Lemma \ref{lm-diagC} to Appendix \ref{Appen-B}.

Next, since the spectral conditions
 $\lambda_s^+<1^{1+\alpha}$  and $1^{1+\alpha} < \lambda_u^-$ hold, we have the classical result of smooth center-stable/center-unstable manifolds for $F$ (see e.g. \cite[Theorem 4.1]{HPS-Book}).

\begin{lemma}\label{IM}
Assume that $F$ is a $C^{1,\alpha}$  diffeomorphism
 satisfying \eqref{Df-hold-1} and \eqref{sp-A}.
Then there exist $C^{1,\alpha}$  center-stable and center-unstable invariant manifolds
$\mathcal{M}_{cs}$ and $\mathcal{M}_{cu}$ of $F$, i.e., there are two $C^{1,\alpha}$ mappings $\varpi_{cs} :X_{cs}\to X_u$ and $\varpi_{cu}:X_{cu}\to X_s$ such that
\begin{align*}
 \mathcal{M}_{cs} =\{x_{cs}+\varpi_{cs}(x_{cs}):x_{cs}\in X_{cs}\},
\quad \mathcal{M}_{cu} =\{x_{cu}+\varpi_{cu}(x_{cu}):x_{cu}\in X_{cu}\},
\end{align*}
where $\varpi_{cs}(x_c)=\varpi_{cu}(x_c)=0$ and $D\varpi_{cs}(x_{cs}), D\varpi_{cu}(x_{cu})$ are globally small such that $D\varpi_{cs}(0)=D\varpi_{cu}(0)=0$.
\end{lemma}
\noindent
By the same arguments as the ones given below Lemma \ref{cen-lem}, we see that the $C^{1,\alpha}$ transformation
\begin{align}\label{GGG}
\mathcal{G}: (x_c,x_s,x_u) \mapsto (x_c, \,x_s-\varpi_{cu}(x_{cu}), \,x_u-\varpi_{cs}(x_{cs}))
\end{align}
can straighten up the center-stable and center-unstable manifolds. Moreover, $\varpi_{cs}(x_c)=\varpi_{cu}(x_c)=0$ means that 
$
\partial_{x_c}\varpi_{cs}(x_c)=\partial_{x_c}\varpi_{cu}(x_c)=0,
$
implying that $\mathcal{G}(x_c)=x_c$
and 
$$
D{\mathcal G}(x_c)=\left(
\begin{array}{ccc}
*& 0& 0
\\
0& *&*
\\
0& *&*
\end{array}
\right),
$$
where $*$ denote elements that may be non-zero,
and $D{\mathcal G}(x)-id$ is globally small such that $D{\mathcal G}(0)=id$. Thus, the transformation $D{\mathcal G}$ does not change \eqref{c-su} and \eqref{CC0}, but it may change \eqref{Df-hold-1} a little such that
\begin{align}\label{Df-hold}
 \|Df(x)\| \le \delta_f, \quad   
 \|D  f(x)-D  f(y)\| \le M \|x-y\|^\alpha,\quad \forall x,y\in\mathbb{R}^d,
\end{align}
where $M>0$ is a constant. Note that the small coefficient $\delta_f>0$ in the second inequality of \eqref{Df-hold-1} is only used to prove Lemma \ref{lm-diagC}, and \eqref{Df-hold} is enough in what follows. After the transformation ${\mathcal G}$, we may assume that 
$\mathcal{M}_{cs}=X_{cs}$ and $\mathcal{M}_{cu}=X_{cu}$,
implying
$
\pi_uF(x_c,x_s,0)=0
$
and
$
\pi_sF(x_c,0,x_u)=0.
$
Hence,
\begin{align}\label{s-u-u-s}
\partial_{x_s}(\pi_uF)(x_c)=0,\quad \partial_{x_u}(\pi_sF)(x_c)=0, \quad
\forall x_c\in X_c,
\end{align}
because $\partial_{x_{cs}}(\pi_uF)(x_c,x_s,0)=0$ and $\partial_{x_{cu}}(\pi_sF)(x_c,0,x_u)=0$
by the fact stated below \eqref{c-su}.

Combining (\ref{c-su}), (\ref{CC0}) with (\ref{s-u-u-s}), we obtain that
\begin{align}\label{block-diag}
DF(x_c)=\left(
\begin{array}{ccc}
\partial_{x_c}(\pi_cF)(x_c)& 0&
0
\\
0& \partial_{x_s}(\pi_sF)(x_c)&
0
\\
0& 0&
\partial_{x_u}(\pi_uF)(x_c)
\end{array}
\right),
\quad \forall x_c\in X_c.
\end{align}
Now, we define $g(x_c):=F(x_c)=\pi_c F(x_c)$ and $A(x_c):=DF(x_c)$, 
and  consider the following linear cocycle  
\begin{align}\label{group}
\mathcal{A}(m,n;x_c):=
\begin{cases}
 A(g^{m-1}(x_c)) \cdots A(g^{n}(x_c)), \quad & m\ge n+1,
\\
id, \quad & m=n,
\\
A(g^m(x_c))^{-1} \cdots A(g^{n-1}(x_c))^{-1}, \quad & m\le n-1.
\end{cases}
\end{align}
By \eqref{Df-hold} and \eqref{block-diag}, we see that
$A(x_c)$ is a small perturbation of $A$, given in \eqref{NL-F}. Thus, applying
the roughness theory of exponential dichotomy (see e.g. \cite{Henry-Book, ZLZ-JDE}), the dichotomy spectrum $\Sigma(A(x_c))$ of $\mathcal{A}(m,n;x_c)$ satisfies
\begin{align}\label{D-sp}
\Sigma(A(x_c))   \subset  \bigcup_{i=1}^p [\lambda_i-\varsigma,\lambda_i+\varsigma]
\cup [1-\varsigma,1+\varsigma]
  \subset
 [\lambda_s^-, \lambda_s^+  ]
   \cup [\lambda_u^-, \lambda_u^+ ]
   \cup  [1-\varsigma, 1+\varsigma ]
\end{align}
for a small constant $\varsigma\ge 0$, where $\lambda_s^-,\lambda_s^+,\lambda_u^-,\lambda_u^+$ are given in (\ref{sp-A}).
Then, by \eqref{block-diag}-\eqref{D-sp}, there is a $K\ge 1$ such that
\begin{align}\label{ED}
\begin{split}
\|\mathcal{A}(m,n;x_c)\pi_s\| \le K (\lambda_s^+)^{m-n}, \quad & \forall m\ge n,
\\
\|\mathcal{A}(m,n;x_c)\pi_u\| \le K (\lambda_u^-)^{m-n}, \quad & \forall m\le n,
\end{split}
\end{align}
and
\begin{align}\label{S-ED}
\begin{split}
\|\mathcal{A}(m,n;x_c)\pi_s\| \le K (\lambda_s^-)^{m-n}, \quad & \forall m\le n,
\\
\|\mathcal{A}(m,n;x_c)\pi_u\| \le K (\lambda_u^+)^{m-n}, \quad & \forall m\ge n.
\end{split}
\end{align}
Moreover,  we see that
\begin{align}\label{g-center}
\begin{split}
 \|\mathcal{A}(m,n;x_c)\pi_c\| \le K (1+\varsigma)^{m-n}, \quad & \forall m\ge n,
\\
\|\mathcal{A}(m,n;x_c)\pi_c\| \le K (1-\varsigma)^{m-n}, \quad & \forall m\le n.
\end{split}
\end{align}
Note that, in the above trichotomy \eqref{ED}-\eqref{g-center}, the projections $\pi_s,\pi_u,\pi_c$ are independent of $x_c$, and this is because the cocycle $\mathcal{A}(m,n;x_c)$ has been block-diagonalized, as seen in \eqref{block-diag}.

Next, we give a result of the stable foliation of $g(x_{cs}):=\pi_{cs}F(x_{cs})$ on the center-stable manifold $X_{cs}$, which will be proved at the end of Section \ref{Sec-4}.

\begin{lemma}\label{IF}
Suppose that $F$ is a $C^{1,\alpha}$ diffeomorphism
 such that  \eqref{Df-hold}, \eqref{block-diag}, \eqref{ED}-\eqref{g-center}  hold. Then
there is a stable foliation of $g:=\pi_{cs}F|_{X_{cs}}$ on $X_{cs}$
with leaves
$$
\mathcal{W}_s(x_{cs})=\{z_s+ h_s(x_{cs},z_s):z_s\in X_s\}, \quad \forall x_{cs}\in X_{cs},
$$
where $h_s:X_{cs} \times X_s \to X_c$ is
$C^{1,\beta}$-smooth with a small constant $\beta>0$ such that
\begin{align}\label{hbars}
h_s(x_{c},0)=x_c\quad {\rm and}\quad Dh_s(x_c,0)=\Pi_{c}, \quad \forall x_c\in X_c,
\end{align}
where %$id_{c}$ is the identity mapping on $X_{c}$.
 $\Pi_c:X_{cs}\times X_s\to X_{c}$ is a projection satisfying
$\Pi_c(x_{cs},z_s)=x_{c}$.
\end{lemma}
\noindent Remind that the $C^{1,\beta}$-smoothness of $h_s$ is basically guaranteed by the bunching condition $(\lambda_s^+)^{\alpha}(1+\varsigma)<1-\varsigma$ with small $\varsigma>0$ for the stable foliation (see \cite[Lemmas 2-3]{ZZJ-MA}), and we just need to show that \eqref{hbars} holds, which will be proved at the end of Section \ref{Sec-4}. 

Lemma \ref{IF} enables us to define a $C^{1,\beta}$ transformation $\phi:X_{cs}\to X_{cs}$ by
\begin{align}\label{phi-cs}
\phi(x_{cs}):=x_{cs}+\rho(x_{cs})(h_s(x_c,x_s)-x_c),\quad \forall x_{cs}=(x_c,x_s)\in X_{cs},
\end{align}
where $\rho:X_{cs}\to \mathbb{R}$ is a smooth cut-off function which is equal to $1$ in a small neighborhood $\tilde U\subset X_{cs}$ of the origin and equal to $0$ outside another small neighborhood $\tilde V\subset X_{cs}$ (containing $\tilde U$) of the origin. It follows from \eqref{hbars} that $\phi(x_c)=x_c$ and $D\phi(x_c)=id_{cs}$, the identity mapping on $X_{cs}$,  for all $x_c\in X_c$. Thus, the new diffeomorphism $(\phi+id_u)^{-1}\circ F \circ (\phi+id_u)$, which is still denoted by $F$ and assumed to have the $C^{1,\alpha}$-smoothness as both $\alpha$ and $\beta$ are small, can be verified to satisfy \eqref{Df-hold} and \eqref{block-diag}. Moreover, since the transformation $\phi$ actually straightens up the stable foliation in $\tilde U\subset X_{cs}$, now we may assume that
\begin{align}\label{Wcs-Stp}
h_s(x_{cs},z_s)=x_c,\quad \forall x_{cs},z_s\in \tilde U.
\end{align}

Finally,  in view of Lemma \ref{lm-diagC},
\eqref{GGG} and \eqref{phi-cs}, we may define $\mathcal{T}:\mathbb{R}^d\to \mathbb{R}^d$ by
\begin{align}\label{TTT}
 \mathcal{T}(x):=  (\phi+id_u)\circ\mathcal{G}^{-1}\circ \Upsilon^{-1} (x)  %\circ \mathcal{L}_b(x)
\end{align}
and verify that $\mathcal{T}$ is a $C^{1,\alpha}$ diffeomorphism satisfying $\mathcal{T}(x_c)=x_c$, $D\mathcal{T}(0)=id$ and 
$D\mathcal{T}(x_c)=D\mathcal{G}^{-1}(x_c) D\Upsilon^{-1}(x_c)$
for  $x_c\in X_c$. Then,
for any $\tilde{x}_c\in X_c$,
\begin{align}\label{DTT}
\|\mathcal{T}(x)-\tilde{x}_c-D\mathcal{T}(\tilde{x}_c)(x-\tilde{x}_c)\|\le L\|x-\tilde{x}_c\|^{1+\alpha}, \quad \forall x\in \mathbb{R}^d,
\end{align}
with a constant $L>0$.
Putting 
$\Delta(\tilde{x}_c):=D\mathcal{T}(\tilde{x}_c)$,
which is an invertible bounded linear operator depending continuously on $\tilde{x}_c$ such that $\Delta(0)=id$, we see from \eqref{DTT} that
\begin{align}\label{T-Re}
\mathcal{T}(x)=\tilde x_c+\Delta(\tilde{x}_c)(x-\tilde{x}_c)+O(\|x-\tilde{x}_c\|^{1+\alpha}).
\end{align}
Similarly, one can obtain that
$$
\mathcal{T}^{-1}(x)=\tilde{x}_c+\Delta(\tilde{x}_c)^{-1}(x-\tilde{x}_c)+O(\|x-\tilde{x}_c\|^{1+\alpha}).
$$

To summary, in what follows of this paper, we always assume that $F$ satisfies \eqref{Df-hold}, \eqref{block-diag} and \eqref{Wcs-Stp}, unless stated otherwise (only in Section \ref{Sec-4} we do not need \eqref{Wcs-Stp}).

%%%%%%%%%%%%%%%%%%%%%%%%%%%%%%%%%%%%%%%%%%%%%%%%%%%%
\section{Regularity of unstable foliation}\label{Sec-4}

In this section, we assume that $F$ satisfies \eqref{Df-hold} and \eqref{block-diag}, while \eqref{Wcs-Stp} is not needed. We investigate  the regularity of the unstable foliation $\{\mathcal{W}_u(x)\}_{x\in \mathbb{R}^d}$ of $F$,
whose leaves can be formulated by
\begin{align}\label{SF-1}
\mathcal{W}_u(x)=\Big\{z\in\mathbb{R}^d: \sup_{n\le 0}\{ \varrho^{-n}\| F^n(z)-F^n(x)\|\}<\infty\Big\},\quad \forall x\in\mathbb{R}^d,
\end{align}
for a constant $\varrho\in (1+\varsigma, \lambda_u^-)$ (see e.g.
 \cite{CHT-JDE}).
Then we have the following result.

\begin{proposition}\label{Fol-thm}
Suppose that $F$ is a $C^{1,\alpha}$ diffeomorphism
 such that \eqref{ED}-\eqref{g-center} hold.
Then there exists an unstable foliation of $F$ with leaves
\begin{align}\label{unst-foli}
\mathcal{W}_u(x)=\{z_u+h_u(x,z_u): z_u\in X_u \},\quad \forall x\in\mathbb{R}^d,
\end{align}
where $h_u: \mathbb{R}^d \times X_u \to X_{cs}$ is continuous in both variables and is $C^{1,\alpha}$-smooth in $z_u$ with a globally bounded derivative.
Furthermore, $h_u$ satisfies that $h_u(x,\pi_u x)=\pi_{cs} x$ and
\begin{align}\label{hh-u}
\|h_u(x,z_u)-\pi_{cs}x\|
\le L\|(\pi_{su} x,z_u)\|^{1+\beta},\quad \forall x\in \mathbb{R}^d,~\forall z_u\in X_u,
\end{align}
where $0< \beta<(\log(1+\varsigma)-\log\lambda_u^-)/\log\lambda_s^-$.
\end{proposition}

In the proof, we use $C,K,L,M$ to denote any positive constants that arise as upper bounds for certain expressions, rather than introducing a more precise indexing $C_i,K_i,L_i,M_i$ for distinct constants.
Moreover, we use $x_*$ instead of $\pi_* x$ with $*=c,s,u,cs,su$ for simplification.

\begin{proof}[Proof of Proposition {\rm \ref{Fol-thm}.}]
To begin with,
we establish the Lyapunov-Perron equation for $F$   along the center manifold.
Let $\mathfrak{U}_\varrho$ consist of all sequences $(q_n)_{n\le 0}\subset \mathbb{R}^d$ such that $\sup_{n\le 0}\varrho^{-n}\|q_n\|<\infty$
for a constant
\begin{align}\label{va-1-u}
\varrho\in (1+\varsigma, \lambda_u^-(\lambda_s^-)^\beta)\subset (1+\varsigma, \lambda_u^-).
\end{align}
The choice of $\varrho$ is possible because $\beta<(\log(1+\varsigma)-\log\lambda_u^-)/\log\lambda_s^-$.
It is a Banach space equipped with the norm $\|\cdot\|_\varrho$ defined by
$
\|\mathbf{q}\|_{\varrho} :=\sup_{n\le 0}\varrho^{-n}\|q_n\|,
$
where $\mathbf{q}:=(q_n)_{n\le 0}\in\mathfrak{U}_\varrho$. Then,
\eqref{SF-1} can be rewritten as
\begin{align}\label{SF-2}
\mathcal{W}_u(x)=\Big\{z\in\mathbb{R}^d: (F^n(z)-F^n(x))_{n\le 0} \in \mathfrak{U}_{\varrho}\Big\}.
\end{align}

Next, recalling that $g$ is defined above \eqref{group}, we give the following lemma.

\begin{lemma}\label{lm-LPeq}
Given any $x\in\mathbb{R}^d$ and $z_u\in X_u$, assume that $(q_n)_{n\le 0}\subset \mathbb{R}^d$ is a sequence such that $\pi_u q_0=z_u-x_u$ and
$(q_n)_{n\le 0}\in \mathfrak{U}_{\varrho}$. Then, $q_n=F^n(x+q_0)- F^n(x)$ for all $n\le 0$
if and only if $(q_n)_{n\le 0}$ satisfies the Lyapunov-Perron equation along the center manifold, i.e.,   %(similarly to  \cite{CHT-JDE})
\begin{align}\label{LP-eqs}
q_n  &= \mathcal{A}(n,0;x_c) (z_u-x_u)
\nonumber\\
&\quad+\sum_{k=n}^{-1} \mathcal{A}(n,k+1;x_c)\pi_u \{f_{g^k(x_c)}(F^k(x)+q_k)-f_{g^k(x_c)}(F^k(x)) \}
\nonumber\\
&\quad-\sum_{k=-\infty}^{n-1}  \mathcal{A}(n,k+1;x_c) \pi_{cs} \{f_{g^k(x_c)}(F^k(x)+q_k)-f_{g^k(x_c)}(F^k(x)) \}
%\quad \forall n\le 0,
\end{align}
for all $n\le 0$,
where $f_{g^k(x_c)}(y):=f(y)-Df(g^k(x_c))y$ for all $y\in\mathbb{R}^d$ and the term
$
\sum_{k=0}^{-1}
$
is defined to be $0$.
\end{lemma}
We will prove this lemma after the completion of this proof.
Subsequently, we claim that if equation \eqref{LP-eqs} has a unique solution $(q_n)_{n\le 0}\in \mathfrak{U}_{\varrho}$, then
\begin{align}\label{hh-qq}
h_u(x,z_u):=x_{cs}+\pi_{cs}q_0
\end{align}
defines the unique unstable foliation of $F$ via \eqref{unst-foli}. In fact, $(q_n)_{n\le 0}\in \mathfrak{U}_{\varrho}$ is a solution of \eqref{LP-eqs} means that $\pi_u q_0=z_u-x_u$ and, therefore,
\begin{align}\label{zuhu}
z_u+h_u(x,z_u)=\pi_uq_0+x_u+x_{cs}+\pi_{cs}q_0=x+q_0
\in \mathcal{W}_u(x)
\end{align}
by \eqref{SF-2} because Lemma \ref{lm-LPeq} implies that
$
(F^n(x+q_0)- F^n(x))_{n\le 0}=(q_n)_{n\le 0}\in \mathfrak{U}_{\varrho}.
$
On the other hand, for a point $z\in \mathcal{W}_u(x)$ such that $\pi_u z=z_u$, we define
$$
\tilde q_n=F^n(z)- F^n(x),\quad \forall n\le 0,
$$
which implies that $(\tilde q_n)_{n\le 0}\in \mathfrak{U}_{\varrho}$ by \eqref{SF-2} and $\pi_u\tilde q_0=z_u-x_u$. Then, Lemma \ref{lm-LPeq} indicates that $(\tilde q_n)_{n\le 0}$ is the unique solution of equation \eqref{LP-eqs} in $\mathfrak{U}_{\varrho}$, i.e., $\tilde q_n=q_n$. It means that
$$
\pi_{cs}z=x_{cs}+\pi_{cs}\tilde q_0=x_{cs}+\pi_{cs}q_0=h_u(x,z_u),
$$
i.e., the point $z$ has the unique form $z_u+h_u(x,z_u)$. Combining this fact with \eqref{zuhu}, we prove the claimed result.

According to the above discussion, in what follows, we need to prove the uniqueness of the solution $\{q_n\}_{n\le 0}$ of equation \eqref{LP-eqs}, and show its regularity with respect to $x$ and $z_u$. For the purpose, recalling that $f_{g^k(x_c)}(y)=f(y)-Df(g^k(x_c))y$ for all $y\in\mathbb{R}^d$, we obtain from
\eqref{Df-hold} that, for all  $n\le 0$,  $Df_{g^n(x_c)}(g^n(x_c))=0$ and
\begin{align}\label{Df-hold2}
\begin{split}
&   \|Df_{g^n(x_c)}(y)\|=\|Df(y)-Df(g^n(x_c))\|\le 2\delta_f,
\\
&\|Df_{g^n(x_c)}(y)-Df_{g^n(x_c)}(\tilde{y})\|=\|Df(y)-Df(\tilde{y})\|\le M\|y-\tilde{y}\|^\alpha
\end{split}
\end{align}
for all $y,\tilde{y}\in \mathbb{R}^d$.
Then, similarly to \cite[Theorems 2.1 and 2.2]{CHT-JDE}, we can prove that
equation \eqref{LP-eqs} has a unique solution $(q_n)_{n\le 0}\in \mathfrak{U}_{\varrho}$, which enables us to regard all $q_n$'s as functions of $(x,z_u)$, and that all $q_n: \mathbb{R}^d\times X_u \to \mathbb{R}^d$ are continuous such that
$$
q_n(x, x_u)=0, \quad \forall n\le 0,
$$
 and $q_n(x,\cdot):X_u \to \mathbb{R}^d$ are $C^{1,\alpha}$ with globally bounded derivatives.

Next, in order to prove \eqref{hh-u}, it suffices to study the regularity of $q_0(x,z_u)$
due to \eqref{hh-qq}. Then, we claim that
\begin{align}\label{Diff-LP}
\sup_{(x,z_u)\in \mathbb{R}^d\backslash \{x_c\}} \frac{\|q_0(x,z_u)-(z_u-x_u) \|}{\|(x_{su},z_u)\|^{1+\beta}}\le L.
\end{align}
In fact, given a constant $\sigma \ge 0$, by \eqref{ED}, \eqref{g-center} and \eqref{Df-hold2},  we see from \eqref{LP-eqs} that
%\begin{scriptsize}
\begin{align}\label{Dif-Fol}
& \sup_{n\le 0}  \frac{\varrho^{-n}\|q_n(x,z_u)-\mathcal{A}(n,0;x_c)(z_u-x_u)\|}{\|(x_{su},z_u)\|^{1+\sigma}}
\nonumber\\
&\le \frac{1}{\|(x_{su},z_u)\|^{1+\sigma}}\sup_{n\le 0} \Bigg\{ \sum_{k=n}^{-1} K (\lambda_u^-/\varrho)^{n-k-1}
\varrho^{-k-1}
           \big\|   f_{g^k(x_c)} (F^k(x)+q_k(x,z_u))- f_{g^k(x_c)} (F^k(x))   \big\|
\nonumber\\
&\quad +\sum_{k=-\infty}^{n-1} K ((1+\varsigma)/\varrho)^{n-k-1}
\varrho^{-k-1} \big\| f_{g^k(x_c)} (F^k(x)+q_k(x,z_u))- f_{g^k(x_c)} (F^k(x))  \big\| \Bigg\}
\nonumber\\
&  \le \frac{C}{\|(x_{su},z_u)\|^{1+\sigma}}
\sup_{k\le 0} \bigg\{\sup_{t\in[0,1]} \varrho^{-k} \| Df_{g^k(x_c)}(F^k(x)+t q_k(x,z_u)) \|\,\|q_k(x,z_u)\|\bigg\}
\nonumber\\
&  =C\sup_{k\le 0} \Bigg\{\sup_{t\in[0,1]} \varrho^{-k} \big\| Df_{g^k(x_c)}(F^k(x)+t q_k(x,z_u)) \big\|
\nonumber\\
&\quad
\cdot\bigg( \frac{\|q_k(x,z_u)\!-\!\mathcal{A}(k,0;x_c)(z_u\!-\!x_u)\|}{\|(x_{su},z_u)\|^{1+\sigma}}
  \!+\! \frac{\|\mathcal{A}(n,0;x_c)(z_u\!-\!x_u)\|}{\|(x_{su},z_u)\|^{1+\sigma}} \bigg)\! \Bigg\}
\nonumber\\
&  \le 2\delta_f C \sup_{k\le 0}   \frac{\varrho^{-k}\|q_k(x,z_u)-\mathcal{A}(k,0;x_c)(z_u-x_u)\|}{\|(x_{su},z_u)\|^{1+\sigma}}
\nonumber\\
& \quad+C \sup_{k\le 0}  \frac{(\lambda_u^-/\varrho)^k\sup_{t\in[0,1]} \big\| Df_{g^k(x_c)}(F^n(x)+t q_n(x,z_u)) \big\| }{\|(x_{su},z_u)\|^{\sigma}}
\end{align}
%\end{scriptsize}
for all $(x,z_u)\in \mathbb{R}^d\backslash\{x_c\}$ (i.e., $\|(x_{su},z_u)\|\ne 0$).
Then, regarding  the case of $\sigma=0$, \eqref{Dif-Fol} gives
\begin{align*}
\sup_{n\le 0} \frac{\varrho^{-n}\|q_n(x,z_u)-\mathcal{A} (n,0;x_c)(z_u-x_u)\|}{\|(x_{su},z_u)\|}
\le \frac{2\delta_fC}{1-2\delta_f C},
\end{align*}
and thus,
\begin{align}\label{Dif-Fo2}
\|q_n(x,z_u)\| \le&~  L\varrho^n\|(x_{su},z_u)\|  +   \|\mathcal{A}(n,0;x_c) (z_u-x_u)\|
\nonumber\\
\le&~ L \varrho^n    \|(x_{su},z_u)\|+ K  (\lambda_u^-)^n \|(z_u-x_u)\|
\le 2L \varrho^n  \|(x_{su},z_u)\|
\end{align}
for all $n\le 0$.

Consider the case that $\sigma =\beta \in (0,\{\log(1+\varsigma)-\log\lambda_u^-\}/\log\lambda_s^-)$.  By \eqref{Dif-Fol} we get
\begin{align}\label{Dif-Fo3}
& \sup_{n\le 0}  \frac{\varrho^{-n}\|q_n(x,z_u)-\mathcal{A}(n,0;x_c)(z_u-x_u)\|}{\|(x_{su},z_u)\|^{1+\beta}}
\nonumber\\
&\le L \sup_{n\le 0}  \frac{(\lambda_u^-/\varrho)^n \sup_{t\in[0,1]} \big\| Df_{g^n(x_c)}(F^n(x)+t q_n) \big\|^{(1-\beta/\alpha)+\beta/\alpha} }{\|(x_{su},z_u)\|^{\beta}}
\end{align}
for all $n\le 0$.
Note that, similarly to \eqref{S-ED}, we have
\begin{align}\label{Dif-Fo5}
\|F^n(x)-g^n(x_c)\|
&= \|F^n(x)-F^n(x_c)\|
\le \sup_{\xi\in \mathbb{R}^d} \|DF^n(\xi)\| \|x_{su}\|
\nonumber\\
&\le (\lambda_s^-)^{n} \|x_{su}\|,\quad \forall n\le 0.
\end{align}
Then, by \eqref{Df-hold2}, \eqref{Dif-Fo2} %\eqref{Dif-Fo4}
and \eqref{Dif-Fo5}, we see from \eqref{Dif-Fo3} that
\begin{align}\label{qq-AA}
& \sup_{n\le 0}  \frac{\varrho^{-n}\|q_n(x,z_u)-\mathcal{A}(n,0;x_c)(z_u-x_u)\|}{\|(x_{su},z_u)\|^{1+\beta}}
\nonumber\\
&\le 2\delta_f^{1-\beta/\alpha} L \sup_{n\le 0}\sup_{t\in[0,1]}\Bigg\{(\lambda_u^-/\varrho)^n
\frac{ \big\| Df_{g^n(x_c)}(F^n(x)+t q_n(x,z_u))-Df_{g^n(x_c)}(g^n(x_c)) \big\|^{\beta/\alpha} }{\big\|F^n(x)+tq_n(x,z_u)- g^n(x_c)\big\|^{\beta}}
\nonumber\\
&\quad \cdot
 \frac{\big\|F^n(x)-g^n(x_c)+tq_n(x,z_u) \big\|^{\beta}}{\|(x_{su},z_u)\|^{\beta}}\Bigg\}
\nonumber\\
& \le 2\delta_f^{1-\beta/\alpha} L M^{\beta/\alpha}
 \sup_{n\le 0}\bigg\{(\lambda_u^-/\varrho)^n  \bigg(\frac{\|F^n(x)-g^n(x_c)\|+\|q_n(x,z_u)\|}{\|(x_{su},z_u)\|}\bigg)^{\beta}\bigg\}
\nonumber\\
& \le 2\delta_f^{1-\beta/\alpha} L M^{\beta/\alpha} \sup_{n\le 0}\{(\lambda_u^-/\varrho)^n
( (\lambda_s^-)^n + 2L\varrho^n )^{\beta}\}
\nonumber\\
&\le L \sup_{n\le 0}\{(\lambda_u^-(\lambda_s^-)^\beta/\varrho)^n\} \le L, \quad \forall n\le 0,
\end{align}
where $ \lambda_u^-(\lambda_s^-)^\beta/\varrho >1$
due to \eqref{va-1-u}. Hence, for any $(x_{cs},z_u)\in \mathbb{R}^d\backslash\{x_c\}$ (i.e., $\|(x_{su},z_u)\|\ne 0$),
\begin{align*}
\frac{\|q_0(x,z_u)-  (z_u-x_u)\|}{\|(x_{su},z_u)\|^{1+\beta}}
 \le \sup_{n\le 0}  \frac{\varrho^{-n}\|q_n(x,z_u)-\mathcal{A}(n,0;x_c)(z_u-x_u)\|}{\|(x_{su},z_u)\|^{1+\beta}} \le L,
\end{align*}
which proves the claimed result \eqref{Diff-LP}.

Finally, we obtain from \eqref{hh-qq} and \eqref{Diff-LP}   that
\begin{align}\label{Dec-9}
\frac{\|h_u(x,z_u)-x_{cs}\|}{\|(x_{su},z_u)\|^{1+\beta}}
&=\frac{\|q_0(x,z_u)- \pi_u q_0(x,z_u)\|}{\|(x_{su},z_u)\|^{1+\beta}}
\nonumber\\
&=\frac{\|q_0(x,z_u)-(z_u-x_u)\|}{\|(x_{su},z_u)\|^{1+\beta}}
\le L.
\end{align}
The proof is completed.
\end{proof}

Remark that in the above proof, we establish the Lyapunov-Perron equation \eqref{LP-eqs} along the center manifold because we need the nonlinearity $f_{g^n(x_c)}$ in \eqref{LP-eqs} to satisfy that $Df_{g^n(x_c)}(g^n(x_c))=0$ (see \eqref{Df-hold2}). This fact enables us to use \eqref{Dif-Fo5} to obtain the estimate
$$
\|F^n(x)-g^n(x_c)\|/\|(x_{su},z_u)\|\le  (\lambda_s^-)^{n}
$$
in the third inequality of \eqref{qq-AA}.
Otherwise, if we use the nonlinearity $f$ from the classical Lyapunov-Perron equation, then the fact $Df(0)=0$ implies that the term $\|F^n(x)-g^n(x_c)\|$ in the second inequality of \eqref{qq-AA} should be replaced with $\|F^n(x)\|$ and, consequently, \eqref{Dif-Fo5} should be replaced with
\begin{align*}
\|F^n(x)\|
\le \sup_{\xi\in\mathbb{R}^d} \|DF^n(\xi)\| \|x\| \le  (\lambda_s^-)^{n} \|x\|,\quad \forall n\le 0.
\end{align*}
This gives rise to a problem that the estimate $\|F^n(x)\|/\|(x_{su},z_u)\|\le  (\lambda_s^-)^{n}$ no longer holds.

Now, we complementally prove Lemma \ref{lm-LPeq}.

\begin{proof}[Proof of Lemma \ref{lm-LPeq}]
Recalling the definition of $f_{g^n(x_c)}$ given in the formulation of Lemma \ref{lm-LPeq}, we have
\begin{align}\label{F-x*}
F(y)=DF(g^n(x_c))y+f_{g^n(x_c)}(y)
\end{align}
for all $n\le 0$ and all $y\in\mathbb{R}^d$.
Then, when $q_n=F^n(x+q_0)- F^n(x)$, we see from \eqref{F-x*} that
\begin{align}\label{qqFF}
q_{n+1}&=F^{n+1}(x+q_0)-F^{n+1}(x)=F(F^n(x+q_0))-F(F^n(x))
\nonumber\\
&=DF(g^n(x_c))q_n+f_{g^n(x_c)}(F^{n}(x)+q_n)
-f_{g^n(x_c)}(F^{n}(x)).
\end{align}
This enables us to apply the discrete Variation of Constant Formula (see e.g. \cite[Theorem 3.17]{Elaydi-book}) to obtain that
\begin{align*}
q_n
=\mathcal{A}(n,0;x_c)q_0
+
  \sum_{k=n}^{-1} \mathcal{A}(n,k+1;x_c)  \{f_{g^k(x_c)}(F^k(x)+q_k)-f_{g^k(x_c)}(F^k(x)) \}
\end{align*}
for all $n\le 0$.
Projecting this equality onto $X_u$ and $X_{cs}$, we obtain that
\begin{align}\label{LP-qqq}
\pi_u q_n &= \mathcal{A}(n,0;x_c) (z_u-x_u)
\nonumber\\
&\quad+\sum_{k=n}^{-1} \mathcal{A}(n,k+1;x_c)\pi_u \{f_{g^k(x_c)}(F^k(x)+q_k)-f_{g^k(x_c)}(F^k(x)) \}
\end{align}
and
\begin{align}\label{Lppp}
\pi_{cs}q_n &=\mathcal{A}(n,0;x_c) \pi_{cs} q_0
\nonumber\\
&\quad+\sum_{k=n}^{-1} \mathcal{A}(n,k+1;x_c)\pi_{cs} \{f_{g^k(x_c)}(F^k(x)+q_k)-f_{g^k(x_c)}(F^k(x)) \},
\end{align}
the second of which is equivalent to
\begin{align}\label{Lppp2}
\pi_{cs} q_0 &=\mathcal{A}(0,n;x_c) \pi_{cs} q_n
\notag \\
&\quad-\sum_{k=n}^{-1} \mathcal{A}(0,k+1;x_c) \pi_{cs}  \{f_{g^k(x_c)}(F^k(x)+q_k)-f_{g^k(x_c)}(F^k(x)) \}.
\end{align}

Since $\mathbf{q}:=(q_n)_{n\le 0}\in\mathfrak{U}_\varrho$, we conclude from \eqref{ED} and \eqref{g-center} that
$$
\|\mathcal{A}(0,n;x_c)\pi_{cs} q_n\| \le K \{(1+\varsigma)/\varrho\}^{-n} \|\mathbf{q}\|_{\varrho},\quad \forall n\le 0,
$$
and from \eqref{Df-hold2} that
\begin{align}\label{uni-con}
&\sum_{k=-\infty}^{-1}  \| \mathcal{A}(0,k+1;x_c) \pi_{cs} \{f_{g^k(x_c)}(F^k(x)+q_k)-f_{g^k(x_c)}(F^k(x))\} \|
\nonumber\\
&\le \sum_{k=-\infty}^{-1} K (1+\varsigma)^{-(k+1)} 2\delta_f \varrho^k \|\mathbf{q}\|_{\varrho}
\le  \delta_f C \, \|\mathbf{q}\|_{\varrho},\quad \forall n\le 0,
\end{align}
as $\varrho>1+\varsigma$.
Therefore, letting $n\to -\infty$ in \eqref{Lppp2}, we have
\begin{align}\label{qq00}
\pi_{cs}q_0 =-\sum_{k=-\infty}^{-1} \mathcal{A}(0,k+1;x_c)
\pi_{cs}
\{f_{g^k(x_c)}(F^k(x)+q_k)-f_{g^k(x_c)}(F^k(x))\},
\end{align}
which together with \eqref{Lppp} yields
\begin{align*}
\pi_{cs} q_n =-\sum_{k=-\infty}^{n-1}  \mathcal{A}(n,k+1;x_c)
\pi_{cs}
\{f_{g^k(x_c)}(F^k(x)+q_k)-f_{g^k(x_c)}(F^k(x))\}.
\end{align*}
Combining this equality with \eqref{LP-qqq}, we obtain that \eqref{LP-eqs} holds for every $n\le 0$.

Conversely, if $(q_n)_{n\le 0}\in\mathfrak{U}_\varrho$ satisfying \eqref{LP-eqs}, which is uniformly convergent by \eqref{uni-con}, then we get
\begin{align*}
q_{n+1}-DF(g^{n}(x_c))q_{n}
&=q_{n+1}-A(g^{n}(x_c))q_{n}
\\
&=
f_{g^n(x_c)}(F^n(x)+q_n)-f_{g^n(x_c)}(F^n(x)).
\end{align*}
It implies that $q_n=F^n(x+q_0)- F^n(x)$ for all $n\le 0$ due to the second equality of \eqref{qqFF}. The proof is completed.
\end{proof}

We complementarily prove Lemma \ref{IF}.

\begin{proof}[Proof of Lemma {\rm \ref{IF}.}]
Let $\mathfrak{U}_{\tilde{\varrho}}$ consist of all sequence ${\bf p}:=(p_n)_{n\ge 0}\subset X_{cs}$ such that $\sup_{n\ge 0}\tilde{\varrho}^{-n}\|p_n\|<\infty$
for a constant $\tilde{\varrho}\in(\lambda_s^+,1-\varsigma)$.
Then, $\mathfrak{U}_{\tilde{\varrho}}$ is a Banach space equipped with norm
$\|\mathbf{p}\|_{\tilde{\varrho}}:=\sup_{n\ge 0}\tilde{\varrho}^{-n} \|p_n\|$.
Similarly to Lemma \ref{lm-LPeq}, one can obtain the following fact:
\begin{itemize}
\item Suppose that $(p_n)_{n\ge 0}\subset X_{cs}$ is a sequence such that $\pi_s p_0=z_s-x_s$ and $(p_n)_{n\ge 0}\in \mathfrak{U}_{\tilde{\varrho}}$.
Then
$p_n=g^n(x_{cs}+p_0)-g^n(x_{cs})$ for all $n\ge 0$ if and only if $(p_n)_{n\ge 0}$ satisfies the Lyapunov-Perron equation along the center manifold, i.e.,
\begin{align}\label{LP-cs}
p_n &=\mathcal{A}_{cs}(n,0;x_c)(z_s-x_s)
\notag \\
&\quad +\sum_{k=0}^{n-1} \mathcal{A}_{cs}(n,k+1;x_c)
\pi_s \{ \tilde f_{g^k(x_{c})}(g^k(x_{cs})+p_k)- \tilde f_{g^k(x_{c})}(g^k(x_{cs})) \}
\notag \\
&\quad - \sum_{k=n}^\infty \mathcal{A}_{cs}(n,k+1;x_c)
 \pi_c \{  \tilde f_{g^k(x_{c})}(g^k(x_{cs})+p_k)- \tilde f_{g^k(x_{c})}(g^k(x_{cs}))\},  
\end{align}
 where $\mathcal{A}_{cs}(n,0;x_c):=\pi_{cs} \mathcal{A}(n,0;x_c)$,
$ \tilde f_{g^k(x_{c})}(y_{cs}):=g(y_{cs})-\pi_{cs} Df(g^k(x_{c}))y_{cs}$ for all $y_{cs}\in X_{cs}$.
\end{itemize}
Since equation \eqref{LP-cs} is merely a transformation of the classical Lyapunov-Perron equation  for the stable foliation (not constructed along the center manifold, see \cite[(3.2)]{ZZJ-MA}), the unique solution $(p_n)_{n\ge 0}\in \mathfrak{U}_{\tilde{\varrho}}$ of \eqref{LP-cs} is also the unique solution of the classical Lyapunov-Perron equation. Thus, we can use \cite[Lemmas 2-3]{ZZJ-MA} to show that $C^{1,1}$-smoothness of $g$ plus bunching condition $\lambda_s^+(1+\varsigma)<1-\varsigma$ guarantees the
$C^{1,\beta}$-smoothness of $p_n$,
and then a non-essential modification of its proof further shows that $C^{1,\alpha}$-smoothness plus the inequality
 $(\lambda_s^+)^\alpha(1+\varsigma)<1-\varsigma$ also ensures the
$C^{1,\beta}$-smoothness of $p_n$.

On the other hand, adopting the same strategy as the proof of Proposition \ref{Fol-thm},
we can prove that equation \eqref{LP-cs} has a unique solution $(p_n(x_{cs},z_s))_{n\ge 0}\in \mathfrak{U}_{\tilde{\varrho}}$ such that
\begin{align}\label{p-reg}
p_n(x_{cs},x_s)=0,\quad
\sup_{n\ge0}\tilde{\varrho}^{-n}\frac{\|p_n(x_{cs},z_s)-\mathcal{A}_{cs}(n,0;x_c)(z_s-x_s)\|}{\|(x_s,z_s)\|^{1+\beta}}\le L
\end{align}
with a small $\beta>0$ when $\|(x_s,z_s)\|\ne 0$.
Moreover,
the stable foliation of $g$ is defined by
\begin{align}\label{stab-fo}
h_s(x_{cs},z_s):=x_c+\pi_c p_0(x_{cs},z_s).
\end{align}
Therefore, it follows from \eqref{p-reg}  and \eqref{stab-fo}    that
\begin{align}\label{hs-1}
h_s(x_{cs},x_s)=x_c+\pi_c p_0(x_{cs},x_s)=x_c
\end{align}
and
\begin{align}\label{hs-2}
\frac{\|h_s(x_{cs},z_s)-x_c\|}{\|(x_s,z_s)\|^{1+\beta}}
&=\frac{\|p_0(x_{cs},z_s)-\pi_s p_0(x_{cs},z_s)\|}{\|(x_s,z_s)\|^{1+\beta}}
= \frac{\|p_0(x_{cs},z_s)-(z_s-x_s)\|}{\|(x_s,z_s)\|^{1+\beta}}
\notag \\
&\le\sup_{n\ge0}\tilde{\varrho}^{-n}\frac{\|p_n(x_{cs},z_s)-\mathcal{A}_{cs}(n,0;x_c)(z_s-x_s)\|}{\|(x_s,z_s)\|^{1+\beta}}\le L
\end{align}
when $\|(x_s,z_s)\|\ne 0$.
Then, for any $\tilde{x}_c\in X_c$, we see from \eqref{hs-1} and \eqref{hs-2} that $h_s(\tilde{x}_c,0)=\tilde{x}_c$ and
\begin{align*}
&\frac{\|h_s(x_{cs},z_s)-h_s(\tilde{x}_c,0)- \Pi_{c}\{(x_{cs},z_s)-(\tilde{x}_c,0)\}\|}{\|(x_{cs},z_s)-(\tilde{x}_c,0)\|}
\notag \\
&= \frac{\|h_s(x_{cs},z_s)-x_c\|}{\|(x_{cs}-\tilde{x}_c,z_s)\|}
%\le   \frac{L\|(x_s,z_s)\|^{1+\beta}}{\|(x_{cs}-\tilde{x}_c,z_s)\|}
\le \frac{L\|(x_s,z_s)\|^{1+\beta}}{\|(x_s,z_s)\|}
= L\|(x_s,z_s)\|^{\beta} \to 0,
\end{align*}
as $ (x_{cs},z_s)\to (\tilde{x}_c,0)$,
since $\|x_s\|\le \max\{\|x_s\|,\|x_c-\tilde x_c\|\}=\|x_{cs}-\tilde{x}_c\|$.
It follows that  $h_s:X_{cs}\times X_s\to X_c$ is differentiable at $(\tilde{x}_c,0)$ with $Dh_s(\tilde{x}_{c},0)=\Pi_c$.
This completes the proof.
\end{proof}

\section{Proof of main result}\label{sec-diff}

In order to prove Theorem \ref{Mai-thm}, in this section, we first introduce a semi-decoupling method by only straightening up the unstable foliation to obtain a new system, which is an expansive fiber-preserving mapping. For the purpose, we define a continuous mapping $H:\mathbb{R}^d \to \mathbb{R}^d$ by
\begin{align}\label{Dec-1}
H(x):=x_u+ h_u(x_{cs}, x_u), \quad \forall x=(x_{cs},x_u)\in\mathbb{R}^d,
\end{align}
and give the following lemma.
\begin{lemma}\label{dec-lem}
Suppose that all conditions of Proposition {\rm \ref{Fol-thm}} hold. Then, the inverse of $H$ exists such that
\begin{align}\label{Dec-2}
H^{-1}(x)=x_u+h_u(x,0), \quad \forall x\in\mathbb{R}^d,
\end{align}
both of which satisfy that $H^{\pm 1}(x_{cs})=x_{cs}$ and
\begin{align}\label{Dif-3}
\|H^{\pm 1}(x)-x\|\le L\|x-\tilde{x}_c\|^{1+\beta},\quad \forall \tilde{x}_c\in X_c.
\end{align}
Moreover,
$
\hat{F}(x):=H^{-1} \circ F \circ H(x)
$
is a continuous mapping such that
\begin{align}\label{xu-F}
\begin{split}
&\pi_{cs} \hat{F}(x)=F(x_{cs}),\quad \|\partial_{x_u}(\pi_u\hat{F})(x)-A_u\|\le \delta,
\\
&\|\partial_{x_u}(\pi_u\hat{F})(x_{cs},x_u)-\partial_{x_u}(\pi_u\hat{F})(x_{cs},\tilde x_u)\|\le M\|x_u-\tilde x_u\|^\beta %\quad \forall x_{cs}\in X_{cs},~x_u,\tilde x_u\in X_u,
\end{split}
\end{align}
for all $x_{cs}\in X_{cs}$ and $x_u,\tilde x_u\in X_u$,
where $\delta>0$ is a small constant and $M>0$ is a constant.
\end{lemma}

Remark that \eqref{Dif-3} means that $H^{\pm 1}$ are differentiable at $\tilde x_c\in \mathcal{M}_c$  with $DH^{\pm 1}(\tilde x_c)=id$ since it implies that
\begin{align}\label{DH0}
\frac{\|H^{\pm 1}(x)-H^{\pm 1}(\tilde x_c)-(x-\tilde x_c)\|}{\|x-\tilde x_c\|^{1+\beta}}=\frac{\|H^{\pm 1}(x)-x\|}{\|x-\tilde x_c\|^{1+\beta}}\le L
\end{align}
when $\|x-\tilde x_c\|\ne 0$.

\begin{proof}[Proof of Lemma {\rm \ref{dec-lem}.}]
In order to show that
$
H^{-1}\circ H=H \circ H^{-1}=id,
$
we note from \eqref{Dec-1} that
$\pi_uH(x)=x_u$ and $H(x)\in \mathcal{W}_u(x_{cs})$, implying that $H(x)$ and $x_{cs}$ belong to the same leaf, i.e.,
\begin{align}\label{huhu}
h_u(H(x),z_u)=h_u(x_{cs},z_u),
\end{align}
and note from Proposition \ref{Fol-thm} that $h_u(x,x_u)=x_{cs}$.
Then, by \eqref{Dec-2} we have
\begin{align*}
H^{-1}(H(x)) =\pi_u H(x) + h_u(H(x),0)
=x_u+h_u(x_{cs},0)=x_u+x_{cs}=x.
\end{align*}
Moreover, noting that $\pi_{cs}H^{-1}(x)=h_u(x,0)$ and $x=x_u+h_u(x,x_u)$ belong to the same leaf, we have
\begin{align*}
H(H^{-1}(x))=x_u+h_u(\pi_{cs}H^{-1}(x),x_u)=x_u+h_u(x,x_u)=x.
\end{align*}
Therefore, $H:\mathbb{R}^d\to \mathbb{R}^d$ is a homeomorphism.

In order to prove \eqref{Dif-3}, by \eqref{Dec-1}-\eqref{Dec-2}, we see that $H^{\pm 1}(x_{cs})=h_u(x_{cs},0)=x_{cs}$.
Moreover, it follows from \eqref{Dec-1} and \eqref{hh-u} (with $x=x_{cs}$) that
\begin{align}\label{Dec-10}
\frac{\|H(x)-x\|}{\|x-x_c\|^{1+\beta}}
=&~  \frac{\|x_u+h_u(x_{cs},x_u)-x\|}{\|x_{su}\|^{1+\beta}}
\notag \\
= &~ \frac{\| h_u(x_{cs},x_u)-x_{cs}\|}{\|(x_s,x_u)\|^{1+\beta}}\le L\quad {\rm for}~x\ne x_c.
\end{align}
Then, for any $\tilde{x}_c\in X_c$, we see from \eqref{Dec-10} that
\begin{align}\label{Dec-11}
\frac{\|H(x)- x\|}{\|x-\tilde{x}_c\|^{1+\beta}}
\le \frac{\|H(x)- x\|}{\|x-x_c\|^{1+\beta}}\le L\quad {\rm for}~x\ne \tilde x_c~{\rm and}~x\ne x_c
\end{align}
since $\|x-x_c\|=\|x_{su}\| \le \|(x_{su}, x_c-\tilde{x}_c)\|= \|x-\tilde{x}_c\|$.
Similarly, for $H^{-1}$, we see from \eqref{Dec-2} and \eqref{hh-u} (with $z_u=0$) that
\begin{align*}
\frac{\|H^{-1}(x)-x\|}{\|x-x_c\|^{1+\beta}}
=&~  \frac{\|x_u+h_u(x,0)-x\|}{\|x_{su}\|^{1+\beta}}
\\
=&~  \frac{\| h_u(x,0)-x_{cs}\|}{\|(x_{su},0)\|^{1+\beta}}\le L\quad {\rm for}~x\ne x_c,
\end{align*}
and then \eqref{Dif-3} holds.

Moreover, by the invariance of the unstable foliation and by \eqref{huhu}, we know that $F\circ H(x)$
and $F(x_{cs})$ belong to the same leaf, i.e.,
$
h_u(F\circ H(x),z_u)=h_u(F(x_{cs}),z_u).
$
Thus,
\begin{align*}
\pi_{cs}\hat F(x)=&~\pi_{cs}H^{-1} \circ F \circ H(x) =h_u(F \circ H(x),0)
\\
=&~ h_u(F(x_{cs}),0)=F(x_{cs}).
\end{align*}
Moreover, since \eqref{Dec-1}-\eqref{Dec-2} give
$
\pi_{u}\hat F(x)=\pi_{u}F \circ H(x)=\pi_{u}F(x_u+ h_u(x_{cs}, x_u)),
$
we compute that
\begin{align*}
\partial_{x_u} (\pi_{u}\hat F)(x)=&~ \Big(\partial_{x_{cs}}(\pi_{u}F)(x_u+ h_u(x_{cs}, x_u))\partial_{x_u} h_u(x_{cs}, x_u),
\\
&\quad  \partial_{x_u}(\pi_{u}F)(x_u+ h_u(x_{cs}, x_u))\Big),
\end{align*}
which together with \eqref{Df-hold} implies \eqref{xu-F}.
This completes the proof.
\end{proof}

Remark that, having Lemma \ref{dec-lem}, we may
define $g:X_{cs}\to X_{cs}$ and $G_{x_{cs}}:X_u\to X_u$ by
\begin{align}\label{gg-GG}
g(x_{cs}):=\pi_{cs} \hat F(x)=F(x_{cs})\quad {\rm and}\quad G_{x_{cs}}(x_u):=\pi_u \hat F(x),
\end{align}
and see from \eqref{xu-F} that
$\hat F(x)=(\pi_{cs}\hat F(x),\pi_u\hat F(x))=(g(x_{cs}), G_{x_{cs}}(x_u))$
can be viewed as a $C^{1,\alpha}$ fiber-preserving mapping (covering $g$) of the fiber bundle $X_{cs}\times X_u$.
Moreover, putting
\begin{align}\label{def-Au}
A_u(x_{cs}):=\partial_{x_u}(\pi_u \hat F)(x_{cs})=DG_{x_{cs}}(0),
\end{align}
we obtain from the first inequality of \eqref{xu-F} that $A_u(x_{cs})$ generates an expansive linear cocycle (over $g$) by the roughness theory of exponential dichotomy (see e.g. \cite{Henry-Book, ZLZ-JDE}). It means that $\hat F$ can be viewed as an expansive $C^{1,\beta}$ fiber-preserving mapping (covering $g$).

Next, let $U\subset \mathbb{R}^d$ be a small neighborhood of the origin such that $\pi_{cs}U\subset \tilde U$, where $\tilde U\subset X_{cs}$ is given in \eqref{Wcs-Stp}. We present the following two lemmas which will be proved in Sections \ref{Sec-6} and \ref{Sec-7}, respectively.
For simplicity of notation, we will use $\beta$ instead of a more precise $\beta_i$ for distinct H\"older exponents, unless stated otherwise.

\begin{lemma}\label{co-redu}
There is a
$C^{1,\beta}$  {\rm (}$\beta\in (0,\alpha)${\rm )}
diffeomorphism $\Theta:U\to \mathbb{R}^d$ such that
$
\tilde{F}:=\Theta\circ \hat F\circ \Theta^{-1}
$
is a $C^{1,\beta}$ diffeomorphism, which satisfies \eqref{xu-F} with $\hat F$ in place of $\tilde{F}$ and satisfies that
\begin{align}\label{barF-Ac}
\partial_{x_u}(\pi_u \tilde{F})(x_{cs})=A_u(x_c),\quad \forall (x_c,x_s)\in \tilde U.
\end{align}
Moreover,
\begin{align}\label{theta-DD}
\Theta(x_{cs})=x_{cs},\quad D\Theta(x_{cs})={\rm diag}(id_{cs}, P_u(x_{cs})),\quad \forall x_{cs}\in \tilde U.
\end{align}
and, for any $\tilde{x}_c\in \tilde U$,
\begin{align}\label{theta-D}
\|\Theta^{\pm 1}(x)-x\|\le L\|x-\tilde{x}_c\|^{1+\beta}, \quad \forall x\in U.
\end{align}
\end{lemma}

By a similar discussion to \eqref{DH0}, we know from \eqref{theta-D} that $D\Theta^{\pm 1}(\tilde x_c)=id$. Then, we can verify that the above $\tilde{F}$ is a $C^{1,\beta}$ diffeomorphism satisfying all conditions of Proposition \ref{Fol-thm}, which gives the unstable foliation of $\tilde{F}$. Without loss of generality, we may assume that the unstable foliation of $\tilde{F}$ is still straightened up because, otherwise, we can use Lemma \ref{dec-lem} again to straighten up the unstable foliation. Note that \eqref{theta-DD} implies that the leaves of the unstable foliation of $\tilde{F}$ are tangent to the subspace $X_u$ at every point $x_{cs}$ since $D\Theta$ does not change the $X_u$ direction at $x_{cs}$. Therefore, the derivative of the transformation that straightens up the unstable foliation is $id$ at $x_{cs}$, i.e., it does not change \eqref{barF-Ac}.

Thus, we still use the notations given in \eqref{gg-GG} to write
\begin{align}\label{barF-FP}
\tilde{F}(x)=(g(x_{cs}), G_{x_{cs}}(x_u)),
\end{align}
where $G_{x_{cs}}(0)=0$ since $X_{cs}$ is still the invariant manifod of $\tilde F$, and give the following lemma.

\begin{lemma}\label{pro-scu}
There exists a homeomorphism $\Phi: U \to \mathbb{R}^d$ such that $\pi_{cs}\Phi(x)=x_{cs}$ and
\begin{align}\label{PSI-F}
\Phi^{-1}\circ \tilde{F}\circ \Phi(x)=(g(x_{cs}),A_u(x_c)x_u),\quad \forall x=(x_{c},x_s, x_u)\in U.
\end{align}
Moreover, for any $\tilde{x}_c\in \tilde U$,
\begin{align}\label{Dscu-9}
\|\Phi^{\pm 1}(x)-x\|\le L\|x-\tilde{x}_c\|^{1+ \beta}, \quad \forall x\in U.
\end{align}
\end{lemma}

Having the above Lemmas \ref{dec-lem}-\ref{pro-scu}, we may define
$
\Psi:=\Phi\circ \Theta^{-1}\circ H,
$
and obtain from \eqref{PSI-F} that
\begin{align}\label{lin-u}
\Psi^{-1}\circ F\circ \Psi(x)=(g(x_{cs}),A_u(x_c)x_u),\quad \forall x=(x_{c},x_s, x_u)\in U,
\end{align}
which satisfies that
\begin{align}\label{H-x}
\begin{split}
\|\Psi(x)-x\|
&\le \|\Phi\circ \Theta^{-1}\circ H(x)-\Theta^{-1}\circ H(x)\|
\\
&\quad +\|\Theta^{-1}\circ H(x)-H(x)\|+\|H(x)-x\|
\\
&\le L\|\Theta^{-1}\circ H(x)-\tilde x_c\|^{1+\beta}
\\
&\quad +L\|H(x)-\tilde x_c\|^{1+\beta}+L\|x-\tilde x_c\|^{1+\beta}
\\
&\le L\Big(\|\Theta^{-1}\circ H(x)\!-\!H(x)\|\!+\!\|H(x)\!-\!x\|\!+\!\|x\!-\!\tilde x_c\|\Big)^{1+\beta}
\\
&\quad+L\Big(\|H(x)-x\|+\|x-\tilde x_c\|\Big)^{1+\beta}+L\|x-\tilde x_c\|^{1+\beta}
\\
&\le L\Big(\|H(x)-\tilde x_c\|^{1+\beta}+\|x-\tilde x_c\|^{1+\beta}+\|x-\tilde x_c\|\Big)^{1+\beta}
\\
&\quad+L\Big(\|x-\tilde x_c\|^{1+\beta}+\|x-\tilde x_c\|\Big)^{1+\beta}+L\|x-\tilde x_c\|^{1+\beta}
\\
&\le 21L\|x-\tilde x_c\|^{1+\beta},
\\
\|\Psi^{-1}(x)\,-&\,x\| \le 21L\|x-\tilde x_c\|^{1+\beta}
\end{split}
\end{align}
for any $\tilde x_c\in \tilde U$ for all $x\in U$ near $\tilde x_c$ by \eqref{Dif-3}, \eqref{theta-D} and \eqref{Dscu-9}. Here, we used the fact that $\|x-\tilde x_c\|^{1+\beta}\le \|x-\tilde x_c\|$ when $\|x-\tilde x_c\|<1$.

Next, we similarly discuss $g(x_{cs})=F(x_{cs})$ on the center-stable manifold $X_{cs}$. 
Note that the stable foliation of $g$ in $\tilde U$ has been straightened up by the discussion given    
below Lemma \ref{IF} 
and that $\partial_{x_s}(\pi_s g)(x_{c}) =A_s(x_c)$.
Thus, we can omit the corresponding results for $g$ to Lemmas \ref{co-redu} and \ref{pro-scu}, and directly apply an argument analogous to that of Lemma \ref{pro-scu} to obtain a homeomorphism
$\psi:\tilde U\to X_{cs}$, which linearizes $g$ along the stable direction, i.e.,
\begin{align}\label{lin-s}
\psi^{-1}\circ g\circ \psi(x_{cs})=(g(x_{c}),A_s(x_c)x_s),\quad \forall x_{cs}=(x_{c},x_s)\in \tilde U,
\end{align}
and satisfies
\begin{align}\label{psi-xcs}
\pi_c\psi(x_{cs})=x_c,\quad \|\psi^{\pm 1}(x_{cs})-x_{cs}\|\le 5L\|x_{cs}-\tilde x_c\|^{1+\beta}
\end{align}
for any $\tilde x_c\in \tilde U$ and for all $x_{cs}\in \tilde U$ near $\tilde x_c$ by using a similar argument to \eqref{H-x}.

Now we are ready to prove our main result.

\begin{proof}[Proof of Theorem {\rm \ref{Mai-thm}.}]
Defining $\tilde \Psi:U\to \mathbb{R}^d$ by
 $\tilde \Psi:=\Psi\circ (\psi,id_u)$, we verify from \eqref{lin-u}, \eqref{lin-s} and the first equality of \eqref{psi-xcs} that
\begin{align*}
 \tilde \Psi^{-1} \circ F \circ \tilde \Psi(x)
 =&~ (\psi,id_u)^{-1} \circ \Psi^{-1} \circ F \circ \Psi \circ (\psi,id_u)(x)
\\
=&~ (\psi^{-1},id_u)\circ \Big(g\circ \psi(x_{cs}), \, A_u(\pi_c \psi(x_{cs}))x_u\Big)
\\
=&~ \Big(\psi^{-1}\circ g \circ \psi(x_{cs}), ~A_u(x_c)x_u\Big)
\\
=&~ \Big(g(x_c), ~A_s(x_c)x_s, ~A_u(x_c)x_u\Big),\quad \forall x\in U,
\end{align*}
and verify from \eqref{H-x} and the second inequality of \eqref{psi-xcs} that
$$
\|\tilde \Psi^{\pm 1}(x)-x\|\le  110L\|x-\tilde x_c\|^{1+\beta}.
$$

Moreover, recalling that ${\mathcal T}$ given in \eqref{TTT} is a $C^{1,\alpha}$ mapping such that ${\mathcal T}(x_c)=x_c$ and \eqref{DTT} holds,
we define
$$
{\mathcal H}:=\tilde \Psi\circ {\mathcal T}.
$$
%where ${\mathcal T}$ is a $C^{1,\alpha}$ mapping given in Remark \ref{Rm-T},
Then, we obtain that
\begin{align}\label{HH-Re}
&\|{\mathcal H}(x)-\tilde x_c-D{\mathcal T}(\tilde x_c)(x-\tilde x_c)\|
\notag \\
&\le \|\tilde \Psi\circ {\mathcal T}(x)-{\mathcal T}(x)\|+\|{\mathcal T}(x)-\tilde x_c-D{\mathcal T}(\tilde x_c)(x-\tilde x_c)\|
\notag \\
&\le 110L\|{\mathcal T}(x)-\tilde x_c\|^{1+\beta}
+L\|x-\tilde x_c\|^{1+\alpha}
\notag \\
&\le 110L\Big(L\|x-\tilde x_c\|^{1+\alpha}+\|D{\mathcal T}(\tilde x_c)\|\,\|x-\tilde x_c\|\Big)^{1+\beta}+L\|x-\tilde x_c\|^{1+\alpha}
\notag \\
&\le M\|x-\tilde x_c\|^{1+\beta},\quad \forall x,\tilde x_c\in U.
\end{align}
Similarly to \eqref{T-Re}, we see from \eqref{HH-Re} that $\mathcal{H}(\tilde{x}_c)=\tilde{x}_c$ and
$$
\mathcal{H}(x)=\tilde{x}_c+\Delta(\tilde x_c)(x-\tilde x_c)+O(\|x-\tilde{x}_c\|^{1+\beta}), \quad \forall x,\tilde x_c\in U,
$$
where $\Delta(\tilde x_c):=D{\mathcal T}(\tilde x_c)$.
Similar conclusion also holds for ${\mathcal H}^{-1}$, which completes the proof.
\end{proof}

%%%%%%%%%%%%%%%%%%%%%%%%%%%%%%%%%%
\section{Reduction of cocycles}\label{Sec-6}

In this section, we prove Lemma \ref{co-redu}. The main idea is first to prove that the cocycles generated by $A_u(x_{cs})$ and $A_u(x_c)$ are cohomologous via a $\beta$-H\"older transfer mapping $P_u:\tilde U\subset X_{cs}\to Gl(d_u,\mathbb{R})$, where $d_u\in\mathbb{N}$ is the dimension of $X_u$, i.e.,
\begin{align}\label{Cohm-eqn}
P_u(g(x_{cs}))A_u(x_{cs})=A_u(x_{c})P_u(x_{cs}).
\end{align}
Here $x \mapsto (x_{cs}, P_u(x_{cs})x_u)$ is called the conjugacy between the two cocycles.
Then, we use the Whitney's extension theorem (see Lemma \ref{Exten-lem} in Appendix \ref{Appen-A}) to find a $C^{1,\beta}$ diffeomorphism $\Theta$ such that $D\Theta(x_{cs})={\rm diag}(id_{cs},P_u(x_{cs}))$, where $id_{cs}$ is the identity mapping on $X_{cs}$.

\begin{proof}[Proof of Lemma {\rm \ref{co-redu}}.]
We divide the proof of this lemma into three steps.

{\em Step 1.
Find a $P_1(x_{cs})\in Gl(d_u,\mathbb{R})$, which is $\beta_{E}$-H\"{o}der with respect to $x_{cs}$ for a small $\beta_E>0$, to diagonalize the linear cocycle generated by
$A_u(x_{cs})$. }

Consider the linear cocycle generated by $A_u(x_{cs}):X_u\to X_u$ defined in \eqref{def-Au}, i.e.,
\begin{align}\label{group-1}
\mathcal{A}_u(m,n;x_{cs}):=
\begin{cases}
 A_u(g^{m-1}(x_{cs})) \cdots A_u(g^{n}(x_{cs})), \, & m\ge n+1,
\\
id, \, & m=n,
\\
A_u(g^m(x_{cs}))^{-1} \cdots A_u(g^{n-1}(x_{cs}))^{-1}, \, & m\le n-1.
\end{cases}
\end{align}
Similarly to the discussion given below \eqref{group}, applying the roughness theory
 of exponential dichotomy, we know that there are projections $\Pi_i(x_{cs}): \mathbb{R}^d\to X_i$ for all $i=k+1,...,p$ corresponding to the decomposition
 $
X_u=E_{k+1}(x_{cs})\oplus\cdots\oplus E_p(x_{cs}),
 $
 where each $E_i(x_{cs})$ is an invariant subspace under the action of $A_u(x_{cs})$,
 such that
\begin{equation}\label{u-dich}
\begin{split}
\|\mathcal{A}_u(m,n;x_{cs})\Pi_i(x_{cs})\| \le K (\lambda_i + \varsigma)^{m-n}, \quad \forall m\ge n,
\\
\|\mathcal{A}_u(m,n;x_{cs})\Pi_i(x_{cs})\| \le K (\lambda_i -\varsigma)^{m-n}, \quad \forall m\le n,
\end{split}
\end{equation}
for all $i=k+1,...,p$. According to \cite[Proposition 3.9]{Pesin}, we see that \eqref{xu-F} and \eqref{u-dich} guarantee that
\begin{align}\label{E-1}
\mathrm{dist} (E_i(x_{cs}),E_i(\tilde{x}_{cs})) \le L\|x_{cs}-\tilde{x}_{cs}\|^{\beta_E}, \quad \forall x_{cs},\tilde x_{cs}\in X_{cs},
\end{align}
for all $i=k+1,...,p$ with a small constant $\beta_E>0$, where the distance between two subspaces $E_1$ and $E_2$
is defined by
\begin{align}\label{dist-E1E2}
\mathrm{dist}(E_1,E_2):=\max\Big(\max_{v_1\in E_1,\|v_1\|=1} \mathrm{dist} (v_1,E_2), \max_{v_2\in E_2,\|v_2\|=1} \mathrm{dist} (v_2,E_1) \Big)
\end{align}
with $\mathrm{dist}(v,E):=\min_{u\in E}\|u-v\|$. It means that each $E_i(x_{cs})$ is $\beta_E$-H\"older continuous with respect to $x_{cs}$.

Then, one can find a $P_1(x_{cs})\in Gl(d_u,\mathbb{R})$, which satisfies
\begin{align}\label{EP-1}
  \|P_1(x_{cs})^{\pm 1}- P_1(\tilde{x}_{cs})^{\pm 1}\| \le L\|x_{cs}-\tilde{x}_{cs}\|^{\beta_E},\quad \forall x_{cs},\tilde x_{cs}\in X_{cs},
\end{align}
such that
\begin{align}\label{PP-cs}
P_1(g(x_{cs})) A_u(x_{cs}) P_1(x_{cs})^{-1}=\mathrm{diag}(A_{k+1}(x_{cs}),...,A_{p}(x_{cs})).
\end{align}
Here, each $A_i(x_{cs}):X_i\to X_i$ is $\beta_E$-H\"{o}lder continuous in $x_{cs}$, i.e.,
\begin{align}\label{AA-u}
  \|A_i(x_{cs})- A_i(\tilde{x}_{cs})\| \le L\|x_{cs}-\tilde{x}_{cs}\|^{\beta_E},\quad \forall x_{cs},\tilde x_{cs}\in X_{cs},
\end{align}
such that
\begin{equation}\label{k-dich}
\begin{split}
\|\mathcal{A}_i(m,n;x_{cs})\| \le K (\lambda_i + \varsigma)^{m-n}, \quad \forall m\ge n,
\\
\|\mathcal{A}_i(m,n;x_{cs})\| \le K (\lambda_i -\varsigma)^{m-n}, \quad \forall m\le n,
\end{split}
\end{equation}
where $\mathcal{A}_i(m,n;x_{cs})$ is the linear cocycle generated by $A_i(x_{cs})$ in the same manner as \eqref{group-1}.

{\em Step 2.
Find a $\beta$-H\"older transfer mapping $P_u:\tilde U\subset X_{cs}\to Gl(d_u,\mathbb{R})$ such that \eqref{Cohm-eqn} holds.}

For all $i=k+1,...,p$ and $x_{cs}=(x_c,x_s)\in \tilde U\subset X_{cs}$, define
\begin{align}\label{B-lim}
B_i(x_{cs}):= \lim_{n\to \infty} \mathcal{A}_i(0,n;x_c) \mathcal{A}_i(n,0;x_{cs}).
\end{align}
If the limit \eqref{B-lim} exists, one can check that
\begin{align}\label{B-conju}
B_i (g(x_{cs})) A_i (x_{cs})
&=\lim_{n\to \infty} \mathcal{A}_i(0,n;g(x_c)) \mathcal{A}_i(n,0;g(x_{cs}))A_i (x_{cs})
\nonumber\\
&=A_i(x_c)\lim_{n\to \infty} \mathcal{A}_i(0,n+1;x_c) \mathcal{A}_i(n+1,0;x_{cs})
\nonumber\\
&=A_i(x_c) B_i(x_{cs}),
\end{align}
where $\pi_c g(x_{cs})=g(x_c)$ since the stable foliation is straightened up in $\tilde U$ as shown in \eqref{Wcs-Stp}.

In order to prove the existence of the limit \eqref{B-lim}, we
note that
\begin{align}\label{BB-conju}
B_i(x_{cs})&=\lim_{n\to \infty} \mathcal{A}_i(0,n;x_c)
\mathcal{A}_i(n,0;x_{cs})
\notag \\
&=\sum_{k=0}^\infty \Big\{ \mathcal{A}_i(0,k+1;x_c)\mathcal{A}_i(k+1,0;x_{cs})-
\mathcal{A}_i(0,k;x_c)\mathcal{A}_i(k,0;x_{cs}) \Big\}
\notag \\
&\quad+id_i,
\end{align}
where $id_i$ is the identity mapping on $X_i$.
Then, by \eqref{AA-u}-\eqref{k-dich} we have
\begin{align}\label{Ga-1}
&\|\mathcal{A}_i(0,k+1;x_c)
\mathcal{A}_i(k+1,0;x_{cs})-
\mathcal{A}_i(0,k;x_c)\mathcal{A}_i(k,0;x_{cs})\|
\notag \\
& \le \|\mathcal{A}_i(0,k+1;x_c) \| \| A_i(g^k(x_{cs}))-A_i(g^k(x_c)) \|
\| \mathcal{A}_i(k,0;x_{cs})\|
\notag \\
& \le K (\lambda_i - \varsigma)^{-(k+1)} L \|g^k(x_{cs})-g^k(x_c)\|^{\beta_E} K(\lambda_i+\varsigma)^k
\notag \\
& \le C \{(\lambda_i+\varsigma)(\lambda_i-\varsigma)^{-1}(\lambda_s^+ )^{\beta_E}\}^k,\quad \forall x_{cs}\in \tilde U,
\end{align}
with a constant $C>0$, where $x_{cs}$ and $x_c$ in $\tilde U$ lie on the same leaf of the stable foliation indicates that $\|g^k(x_{cs})-g^k(x_c)\| \le K (\lambda_s^+)^k$ for all $k\ge 0$ by a definition for the stable foliation similar to \eqref{SF-1}.
Thus, $B_i(x_{cs})$ is well-defined such that
\begin{align}\label{BB-idid}
B_i(x_{c})=id_i,\quad \forall x_c\in \tilde U,
\end{align}
since
$(\lambda_i+\varsigma)(\lambda_i-\varsigma)^{-1} (\lambda_s^+)^{\beta_E}<1$ with small $\varsigma>0$.

Next, we show that $B_i(x_{cs})$ is  H\"{o}lder continuous in $x_{cs}$.
For any $x_{cs}=(x_c,x_s),\tilde{x}_{cs}=(\tilde x_c,\tilde x_s)\in \tilde{U}$, by \eqref{BB-conju} we have
\begin{align*}
B_i(x_{cs})-B_i(\tilde{x}_{cs})
& = \sum_{k=0}^\infty \Big\{ \mathcal{A}_i(0,k+1;x_c)
\mathcal{A}_i(k+1,0;x_{cs})-
\mathcal{A}_i(0,k;x_c)\mathcal{A}_i(k,0;x_{cs})
\notag \\
&\qquad -\mathcal{A}_i (0,k+1;\tilde{x}_c)\mathcal{A}_i(k+1,0;\tilde{x}_{cs})
+\mathcal{A}_i(0,k;\tilde{x}_c)
\mathcal{A}_i(k,0;\tilde{x}_{cs})  \Big\}
\notag \\
& = \sum_{k=0}^\infty  \Gamma_k(x_{cs}) -\Gamma_k(\tilde{x}_{cs}),
\end{align*}
where
$
\Gamma_k(x_{cs}):=\mathcal{A}_i(0,k+1;x_c) \mathcal{A}_i(k+1,0;x_{cs})
- \mathcal{A}_i(0,k;x_c) \mathcal{A}_i(k,0;x_{cs})
$
satisfies that
$$
\|\Gamma_k(x_{cs})\| \le C \{(\lambda_i+\varsigma)(\lambda_i-\varsigma)^{-1}(\lambda_s^+ )^{\beta_E}\}^k,\quad \forall x_{cs}\in \tilde U,
$$
due to \eqref{Ga-1}. Moreover, we note that
\begin{align*}
 \Gamma_k(x_{cs})-\Gamma_k(\tilde{x}_{cs})
=&~ \mathcal{A}_i(0,k+1;x_c)
   \big( \mathcal{A}_i(k+1,0;x_{cs})-\mathcal{A}_i(k+1,0;\tilde{x}_{cs})\big)
\\
&~+ \big( \mathcal{A}_i(0,k+1;x_c)-\mathcal{A}_i(0,k+1;\tilde{x}_c)\big) \mathcal{A}_i(k+1,0;\tilde{x}_{cs})
\\
&~+ \mathcal{A}_i(0,k;x_c)
   \big( \mathcal{A}_i(k,0;x_{cs})-\mathcal{A}_i(k,0;\tilde{x}_{cs})\big)
\\
&~+ \big( \mathcal{A}_i(0,k;x_c)-\mathcal{A}_i(0,k;\tilde{x}_c)\big) \mathcal{A}_i(k,0;\tilde{x}_{cs})
\\
=:&~ \Gamma_{k,1}(x_{cs},\tilde{x}_{cs}) + \Gamma_{k,2}(x_{cs},\tilde{x}_{cs}) + \Gamma_{k,3}(x_{cs},\tilde{x}_{cs}) + \Gamma_{k,4}(x_{cs},\tilde{x}_{cs}).
\end{align*}
Then, \eqref{AA-u}-\eqref{k-dich} give
\begin{align*}
 \|\Gamma_{k,1}(x_{cs},\tilde{x}_{cs})\|
=&~ \|\mathcal{A}_i(0,k+1;x_c)
   \big( \mathcal{A}_i(k+1,0;x_{cs})-\mathcal{A}_i(k+1,0;\tilde{x}_{cs})\big)\|
\\
 \le&~  \|\mathcal{A}_i(0,k+1;x_c)\|
\|A_i(g^k(x_{cs}))\circ \cdots \circ A_i(x_{cs})
\\
&\quad- A_i(g^k(\tilde{x}_{cs}))\circ \cdots \circ A_i(\tilde{x}_{cs}) \|
\\
 \le &~ \|\mathcal{A}_i(0,k+1;x_c)\| \sum_{i=0}^k
\|A_i(g^k(x_{cs})) \cdots  A_i(g^{i+1}(x_{cs}))
\\
&\quad  \cdot\big(A_i(g^i(x_{cs}))-A_i(g^i(\tilde{x}_{cs})) \big)
   A_i(g^{i-1}(\tilde{x}_{cs})) \cdots A_i(\tilde{x}_{cs}) \|
\\
\le&~ K (\lambda_i-\varsigma)^{-k-1} \sum_{i=0}^k K (\lambda_i+\varsigma)^k L\|g^i(x_{cs})-g^i(\tilde{x}_{cs})\|^{\beta_E}
\\
\le&~ K^3L (\lambda_i-\varsigma)^{-k}(\lambda_i+\varsigma)^{k}
\sum_{i=0}^k (1+\varsigma)^{i \beta_E}
\|x_{cs}-\tilde{x}_{cs}\|^{\beta_E}
\\
\le &~ C (\lambda_i-\varsigma)^{-k}(\lambda_i+\varsigma)^{k} (1+2\varsigma)^{k \beta_E}\|x_{cs}-\tilde{x}_{cs}\|^{\beta_E},
\end{align*}
and similarly,
\begin{align*}
\|\Gamma_{k,2}(x_{cs},\tilde{x}_{cs})\| \le &~
C (\lambda_i-\varsigma)^{-k}(\lambda_i+\varsigma)^{k} (1+2\varsigma)^{k \beta_E }\|x_{c}-\tilde{x}_{c}\|^{\beta_E},
\\
\|\Gamma_{k,3}(x_{cs},\tilde{x}_{cs})\| \le &~ C (\lambda_i-\varsigma)^{-k}(\lambda_i+\varsigma)^{k} (1+2\varsigma)^{k \beta_E}\|x_{cs}-\tilde{x}_{cs}\|^{\beta_E},
\\
\|\Gamma_{k,4}(x_{cs},\tilde{x}_{cs})\| \le &~
C (\lambda_i-\varsigma)^{-k}(\lambda_i+\varsigma)^{k} (1+2\varsigma)^{k \beta_E}\|x_{c}-\tilde{x}_{c}\|^{\beta_E}.
\end{align*}
Therefore, in view of the above inequalities, we have
$$
\| \Gamma_k(x_{cs})-\Gamma_k(\tilde{x}_{cs}) \| \le 4C (\lambda_i-\varsigma)^{-k}(\lambda_i+\varsigma)^{k} (1+2\varsigma)^{k \beta_E}\|x_{cs}-\tilde{x}_{cs}\|^{\beta_E}.
$$

Next, let
$$
\tau_1 := \left(\frac{\lambda_i + \varsigma}{\lambda_i - \varsigma}\right) (\lambda_s^ +)^{\beta_E}<1, \quad
\tau_2 := \left(\frac{\lambda_i + \varsigma}{\lambda_i - \varsigma}\right) (1+2\varsigma)^{\beta_E}, \quad \rho=1.
$$
Then, since $\rho\tau_1<1$,  applying  Lemma \ref{Holder-lem} in Appendix \ref{Appen-A}, we get
\begin{align}\label{B-Holder}
\|B_{i}(x_{cs})-B_{i}(\tilde{x}_{cs})\|\le \sum_{k=0}^\infty \|\Gamma_k(x_{cs})-\Gamma_k(\tilde{x}_{cs})\| \le C \|x_{cs}-\tilde{x}_{cs}\|^{\beta_i},
%\quad \forall x_{cs},\tilde{x}_{cs}\in \tilde{U},
\end{align}
for any $x_{cs},\tilde{x}_{cs}\in \tilde{U}$,
with small constants $\beta_i>0$ for all $i=k+1,...,p$.
Now, we define
$
B(x_{cs}):=(B_{k+1}(x_{cs}),...,B_p(x_{cs}))
$
and
$$
P_u(x_{cs}):=B(x_{cs})P_1(x_{cs})P_1(x_c)^{-1}, \quad \forall x_{cs}\in \tilde{U}.
$$
By \eqref{PP-cs} and \eqref{B-conju}, we see that $P_u(x_{cs})$ satisfies \eqref{Cohm-eqn}, and
by \eqref{EP-1}, \eqref{BB-idid} and \eqref{B-Holder}, we see that $P_u(x_{cs})$ is $\beta$-H\"older with respect to $x_{cs}$ for a small constant $\beta>0$ and that $P_u(x_{c})=id_u$.

{\em Step 3. Applying the Whitney's extension theorem to obtain $\Theta$.}

In this step, without loss of  generality, we assume that $X_{cs}$ and $X_u$ are 1-dimensional spaces and that $x_{cs}$ and $x_u$ correspond to $z_1$ and $z_2$, respectively, given in the beginning of Appendix \ref{Appen-A}.

Defining $\theta_{00}(x_{cs}):=x_{cs}$, $\theta_{10}(x_{cs}):=1$ (i.e., $id_{cs}$)
and $\theta_{01}(x_{cs}):=0$
for all $x_{cs}\in \tilde U\subset X_{cs}$, we verify from \eqref{W.1} in Appendix \ref{Appen-A} (with $\Omega=\tilde U$, $r=1$, ${\bf j}, {\bf k}=(0,0),(1,0),(0,1)$) that
\begin{align*}
&R_{00}(x_{cs},\tilde x_{cs})=\theta_{00}(x_{cs})-\theta_{00}(\tilde x_{cs})-\theta_{10}(\tilde x_{cs})(x_{cs}-\tilde x_{cs})=0,
\\
&R_{10}(x_{cs},\tilde x_{cs})=\theta_{10}(x_{cs})-\theta_{10}(\tilde x_{cs})=0,\quad
R_{01}(x_{cs},\tilde x_{cs})=0,
\end{align*}
which implies that  \eqref{W.2} in Appendix \ref{Appen-A} holds when $w_{\bf j}=\theta_{\bf j}$.
Moreover, defining $\vartheta_{00}(x_{cs}):=0$, $\vartheta_{10}(x_{cs}):=0$
and $\vartheta_{01}(x_{cs}):=P_u(x_{cs})$
for all $x_{cs}\in \tilde U\subset X_{cs}$, we verify that
\begin{align*}
&R_{00}(x_{cs},\tilde x_{cs})=0,\quad
R_{10}(x_{cs},\tilde x_{cs})=0,
\\
&R_{01}(x_{cs},\tilde x_{cs})=\vartheta_{01}(x_{cs})-\vartheta_{01}(\tilde x_{cs})=P_u(x_{cs})-P_u(\tilde x_{cs}),
\end{align*}
which implies that  \eqref{W.2} in Appendix \ref{Appen-A} holds with $\alpha=\beta$ when $w_{\bf j}=\vartheta_{\bf j}$ since $P_u(x_{cs})$ is $\beta$-H\"older with respect to $x_{cs}$, as seen at the end of Step 2. Hence, by Lemma \ref{Exten-lem} (the Whitney’s extension theorem) given in Appendix \ref{Appen-A}, we obtain a $C^{1,\beta}$ mapping $\Theta: U\subset \mathbb{R}^d\to \mathbb{R}^d$ such that
\begin{align*}
&\pi_{cs} \Theta(x_{cs})=x_{cs},
\quad
\pi_u \Theta(x_{cs})=0,
\\
& D(\pi_{cs} \Theta)(x_{cs})=(id_{cs},0),
\quad
D(\pi_u \Theta)(x_{cs})=(0, P_u(x_{cs})),
\end{align*}
i.e., \eqref{theta-DD} holds.
Furthermore, since $P_u(x_{c})=id_u$ (i.e., $D\Theta(x_c)=id$), as seen at the end of Step 2, we conclude that $\Theta: U\subset \mathbb{R}^d\to \mathbb{R}^d$ is a $C^{1,\beta}$ diffeomorphism such that
\begin{align*}
\|\Theta^{\pm 1}(x)-x\|&=\|\Theta^{\pm 1}(x)-\Theta^{\pm 1}(\tilde x_c)-(x-\tilde x_c)\|
\\
&\le \|D\Theta^{\pm 1}(\xi)-D\Theta^{\pm 1}(\tilde x_c)\|\,\|x-\tilde x_c\|\le L\|x-\tilde x_c\|^{1+\beta},
\end{align*}
where $\xi$ lies between $x$ and $\tilde x_c$. This proves \eqref{theta-D}. The discussion for the general case that $X_{cs}$ and $X_u$ are higher dimensional spaces is almost the same.

Finally, by \eqref{gg-GG}-\eqref{def-Au}, we obtain that
$$
D\hat F(x_{cs})=
\left(
\begin{array}{ccc}
Dg(x_{cs})& 0
\\
* & A_u(x_{cs})
\end{array}
\right)
$$
and therefore \eqref{theta-DD} yields
\begin{align*}
D\tilde{F}(x_{cs})
=&~D\Theta(\hat F\circ \Theta^{-1}(x_{cs}))D\hat F(\Theta^{-1}(x_{cs}))D\Theta^{-1}(x_{cs})
\\
=&~\left(
\begin{array}{ccc}
id_{cs}& 0
\\
0 & P_u(g(x_{cs}))
\end{array}
\right)
\left(
\begin{array}{ccc}
Dg(x_{cs})& 0
\\
* & A_u(x_{cs})
\end{array}
\right)
\left(
\begin{array}{ccc}
id_{cs}& 0
\\
0 & P_u(x_{cs})^{-1}
\end{array}
\right)
\\
=&~\left(
\begin{array}{ccc}
Dg(x_{cs})& 0
\\
* & P_u(g(x_{cs}))A_u(x_{cs})P_u(x_{cs})^{-1}
\end{array}
\right)
\\
=&~
\left(
\begin{array}{ccc}
Dg(x_{cs})& 0
\\
* & A_u(x_{c})
\end{array}
\right),
\end{align*}
which proves \eqref{barF-Ac}. The proof is completed.
\end{proof}

Remark that $A_u(x_{cs})$ cannot be completely reduced to $A_u$, because once $A_u(x_c)$ is replaced by $A_u$, the term $\lambda_s^+<1$ in the estimate \eqref{Ga-1} would be replaced by $1+\varsigma>1$, and then we cannot guarantee that the convergence of the limit defined in \eqref{BB-conju}.

\section{Linearization of expansive fiber-preserving mapping}\label{Sec-7}

In this section, we are ready to prove  Lemma  \ref{pro-scu}, where we recall that the expensive fiber-preserving mapping $\tilde{F}$ has the form of \eqref{barF-FP}.

\begin{proof}[Proof of Lemma {\rm \ref{pro-scu}}.]
We divide the proof of this lemma into three steps.

{\em Step 1.
Lift $\tilde{F}$ to a mapping $\mathbb{G}_u$ defined in a space of bounded continuous mappings,
and then prove the differentiable linearization for $\mathbb{G}_u$.
}

Let $\mathcal{X}_\infty:=C^0_b(X_{cs},X_u)$ be the set of all bounded continuous mappings satisfying
$$
\|\eta\|_\infty:=\sup_{x_{cs}\in X_{cs}} \|\eta(x_{cs})\|<\infty, \quad \forall \eta\in \mathcal{X}_\infty.
$$
Then, $(\mathcal{X}_\infty,\|\cdot\|_\infty)$ is a Banach space.
Define a linear operator $\mathbb{A}_u: \mathcal{X}_\infty\to \mathcal{X}_\infty$ by
\begin{align}\label{Auu-g}
(\mathbb{A}_u \eta )(x_{cs}): =&~ \partial_{x_u}(\pi_u \tilde{F})(g^{-1}(x_{cs}))\eta(g^{-1}(x_{cs}))
\notag \\
=&~A_u(\pi_c g^{-1}(x_{cs}))\eta(g^{-1}(x_{cs}))= A_u(g^{-1}(x_{c}))  \eta(g^{-1}(x_{cs})) %\quad \forall \eta\in \mathcal{X}_\infty,
\end{align}
due to \eqref{barF-Ac}, where $\pi_c g^{-1}(x_{cs})=g^{-1}(x_c)$, as known below \eqref{B-conju}.
It follows from \eqref{ED} and  \eqref{S-ED} that $\mathbb{A}_u$ is a well-defined and bounded linear operator.
  Furthermore, we can verify that (i) the zero mapping $\mathbf{0} \in \mathcal{X}_\infty$ is a  fixed point of $\mathbb{A}_u \eta$;
(ii) $\mathbb{A}_u \eta$ is invertible and
$$
(\mathbb{A}_u^{-1} \eta)(x_{cs}):= A_u(x_{c})^{-1} \eta(g(x_{cs})), \quad \forall \eta\in \mathcal{X}_\infty.
$$
Using similar discussion for the dichotomy spectrum given in \cite[Section 3]{DZZ-MZ}, we can deduce from \eqref{D-sp} that
the spectral $\Sigma(\mathbb{A}_u)$ satisfies
\begin{align}\label{u-sp}
\Sigma(\mathbb{A}_u)\subset \bigcup_{i=k+1}^p [\lambda_i-\varsigma,\lambda_i+\varsigma].
\end{align}

Next, putting
\begin{align*}
\tilde{f}_{x_{cs}}(x_u):=G_{x_{cs}}(x_u)-A_u(x_{cs})x_u
=\pi_u \tilde F(x)-\partial_{x_u}(\pi_u\tilde F)(x_{cs})x_u,
\end{align*}
we verify that $\tilde{f}_{x_{cs}}(0)=G_{x_{cs}}(0)=0$ and
\begin{align}\label{DG0A}
  D\tilde{f}_{x_{cs}}(0)=\partial_{x_u}(\pi_u \tilde F)(x_{cs})-\partial_{x_u}(\pi_u\tilde F)(x_{cs})=0,
\end{align}
as seen from \eqref{barF-FP}.
This enables us to
define a mapping $\mathbb{G}_u: \mathcal{X}_\infty  \to \mathcal{X}_\infty$ by
\begin{align}\label{Fuu}
 \mathbb{G}_u (\eta)(x_{cs})
 := &~G_{g^{-1}(x_{cs})}(\eta(g^{-1}(x_{cs})) )
 \nonumber\\
 = &~A_u(g^{-1}(x_{cs}))  \eta(g^{-1}(x_{cs}))+ \tilde{f}_{g^{-1}(x_{cs})}(\eta(g^{-1}(x_{cs})) ).
 % \,\, \forall   \eta\in \mathcal{X}_\infty.
\end{align}
%Now, we establish the following auxiliary results for $\mathbb{G}_u$.
For any $\eta\in \mathcal{X}_\infty$, we obtain from \eqref{xu-F} (with $\hat F$ in place of $\tilde F$) and   \eqref{u-sp}-\eqref{DG0A} that
\begin{align*}
\sup_{x_{cs}\in X_{cs}} \|\mathbb{G}_u (\eta)(x_{cs})\|
&=\sup_{x_{cs}\in X_{cs}} \big\| A_u(g^{-1}(x_{cs}))  \eta(g^{-1}(x_{cs}))+ \tilde{f}_{g^{-1}(x_{cs})}(\eta(g^{-1}(x_{cs})) ) \big\|
\\
&\le(\lambda_p+\varsigma + \delta) \|\eta\|_\infty,
\end{align*}
which means that $\mathbb{G}_u:\mathcal{X}_\infty  \to \mathcal{X}_\infty$ is well-defined such that $\mathbb{G}_u({\bf 0})={\bf 0}$.

For the differentiability of $\mathbb{G}_u$, we claim that
\begin{align*}
(D\mathbb{G}_u(\eta)\sigma) (x_{cs})
=A_u(g^{-1}(x_{cs}))\sigma(g^{-1}(x_{cs}))
+ D\tilde{f}_{g^{-1}(x_{cs})}(\eta(g^{-1}(x_{cs}))) \sigma(g^{-1}(x_{cs}))
\end{align*}
for any $\eta,\sigma\in \mathcal{X}_\infty$.
In fact, by \eqref{Fuu} we have
\begin{align*}
&\mathbb{G}_u(\eta+\sigma)(x_{cs}) - \mathbb{G}_u(\eta)(x_{cs}) -(D\mathbb{G}_u(\eta)\sigma)(x_{cs})
\\
& = \tilde{f}_{g^{-1}(x_{cs})}((\eta+\sigma)(g^{-1}(x_{cs})) ) - \tilde{f}_{g^{-1}(x_{cs})}(   \eta(g^{-1}(x_{cs})) )
\\
&\quad  - D\tilde{f}_{g^{-1}(x_{cs})}(\eta(g^{-1}(x_{cs}))) \sigma(g^{-1}(x_{cs}))
\\
& =\int_0^1 \Big( D\tilde{f}_{g^{-1}(x_{cs})}((\eta+t\sigma)(g^{-1}(x_{cs}))) \sigma(g^{-1}(x_{cs}))
\\
&\qquad  - D\tilde{f}_{g^{-1}(x_{cs})}(\eta(g^{-1}(x_{cs})))
    \sigma(g^{-1}(x_{cs})) \Big) \, dt.
\end{align*}
Then by  \eqref{xu-F}, we further have
\begin{align*}
&\| \mathbb{G}_u(\eta+\sigma) - \mathbb{G}_u(\eta) - D(\mathbb{G}_u(\eta))\sigma \|_\infty
\\
&\le \sup_{x_{cs}\in X_{cs}} \int_0^1
\big\| D\tilde{f}_{g^{-1}(x_{cs})}( (\eta+t\sigma)(g^{-1}(x_{cs})))
\\
&\qquad - D\tilde{f}_{g^{-1}(x_{cs})}(\eta(g^{-1}(x_{cs}))) \big\| \| \sigma(g^{-1}(x_{cs})) \|  \, dt
\\
&\le \sup_{x_{cs}\in X_{cs}} \int_0^1 M \|t \sigma(g^{-1}(x_{cs})) \|^{\beta}   \| \sigma(g^{-1}(x_{cs})) \|  \, dt \le M \|\sigma\|_\infty^{1+\beta},
\end{align*}
which implies that
$$
\lim_{\sigma\to {\bf 0}} \frac{\| \mathbb{G}_u(\eta+\sigma) - \mathbb{G}_u(\eta) - D(\mathbb{G}_u(\eta))\sigma \|_\infty}{\|\sigma\|_\infty} = 0.
$$
This proves the differentiability of $\mathbb{G}_u$ such that
\begin{align*}
(D\mathbb{G}_u({\bf 0})\sigma) (x_{cs})
&=A_u(g^{-1}(x_{cs}))\sigma(g^{-1}(x_{cs})) + D\tilde{f}_{g^{-1}(x_{cs})}(0) \sigma(g^{-1}(x_{cs}))
\nonumber\\
&=\partial_{x_u}(\pi_u\tilde F)(g^{-1}(x_{cs}))\sigma(g^{-1}(x_{cs}))=(\mathbb{A}_u \sigma)(x_{cs}),
\end{align*}
implying that $D\mathbb{G}_u({\bf 0})=\mathbb{A}_u$, because of \eqref{Auu-g}
and \eqref{DG0A}.

Moreover, by \eqref{xu-F} again we have
\begin{align*}
\|D\mathbb{G}_u(\eta_1)\sigma-D\mathbb{G}_u(\eta_2)\sigma\|_\infty
&= \sup_{x_{cs}\in X_{cs}} \big\| D\tilde{f}_{g^{-1}(x_{cs})}(\eta_1(g^{-1}(x_{cs})))   \sigma(g^{-1}(x_{cs}))
\\
&\qquad   - D\tilde{f}_{g^{-1}(x_{cs})}(\eta_2(g^{-1}(x_{cs})))   \sigma(g^{-1}(x_{cs})) \big\|
\\
&\le M \sup_{x_{cs}\in X_{cs}} \| \eta_1(g^{-1}(x_{cs})) -\eta_2(g^{-1}(x_{cs}))\|^\beta \|\sigma(g^{-1}(x_{cs}))\|
\\
&\le M \|\eta_1-\eta_2\|_\infty^\beta \|\sigma\|_\infty, \quad \forall \eta_1,\eta_2, \sigma\in \mathcal{X}_\infty,
\end{align*}
and
\begin{align*}
\|D\mathbb{G}_u(\eta)\sigma-\mathbb{A}_u\sigma\|_\infty
= \sup_{x_{cs}\in X_{cs}}   \big\| D\tilde{f}_{g^{-1}(x_{cs})}(\eta(g^{-1}(x_{cs})))   \sigma(g^{-1}(x_{cs})) \big\|
\le \delta \|\sigma\|_\infty.
\end{align*}
Thus, $\mathbb{G}_u$ is $C^{1,\beta}$ whose nonlinearity has globally small bound.

Now, we obtain a result on differentiable linearization for $\mathbb{G}_u$.

\begin{lemma}\label{lem-u}
Let $\mathbb{G}_u$ and $\mathbb{A}_u$ be given above. Then there exists a local homeomorphism $\mathbb{H}_u$ on  $\mathcal{X}_\infty$  such that $\mathbb{H}_u(0)=0$  and
\begin{align}\label{us-conj}
\mathbb{H}_u^{-1} \circ \mathbb{G}_u \circ  \mathbb{H}_u= \mathbb{A}_u.
\end{align}
Moreover, for any $\eta\in \mathcal{X}_\infty$ near the origin,
\begin{align}\label{HH-u}
\|\mathbb{H}_u^{\pm 1}(\eta)-\eta\|_\infty \le L \|\eta\|_\infty^{1+
\beta}
 %\quad \rm{where}~\beta\in (0,\alpha).
\end{align}
with a small constant $\beta>0$.
\end{lemma}

\begin{proof}[Proof of Lemma {\rm \ref{lem-u}}.]
Note that the spectral bandwidth condition 
$$
(\lambda_i + \varsigma)/(\lambda_i-\varsigma)< (\lambda_{k+1}-\varsigma)^\beta
$$ 
holds automatically for every $i=k+1,\cdots,p$.
Then, one can apply \cite[Lemma 6.1]{ZLZ-TAMS} to obtain the local $C^0$ linearization for $\mathbb{G}_u$,
  where the conjugacy $\mathbb{H}_u$ and its inverse $\mathbb{H}_u^{-1}$ satisfy   \eqref{HH-u}.
This completes the proof.
\end{proof}

{\em Step 2.
Construct a transformation $\Phi$ via $\mathbb{H}_u$ and verify its invertibility.
}

Fixing arbitrarily a point $x=(x_{cs},x_u)\in \mathbb{R}^d$ and a mapping
$\zeta_x \in \mathcal{X}_\infty$ such that $\zeta_x(x_{cs}) = x_u$,
we define $\Phi:\mathbb{R}^d\to \mathbb{R}^d$ by
\begin{align}\label{P-H1}
\Phi(x):=(x_{cs},\mathbb{H}_u(\zeta_x)(x_{cs})),
\end{align}
which is continuous since both $\mathbb{H}_u$ and $\zeta_x$ are continuous.
Then, by \eqref{Auu-g}, \eqref{Fuu}, \eqref{us-conj} and \eqref{P-H1}, one can verify that
\begin{align}\label{verf-conj}
\tilde{F}\circ \Phi(x)=&~ \tilde{F} (x_{cs},\mathbb{H}_u(\zeta_x)(x_{cs}))
= (g(x_{cs}),G_{x_{cs}}(\mathbb{H}_u(\zeta_x)(x_{cs})))
\nonumber\\
=&~ (g(x_{cs}),\mathbb{G}_u \circ\mathbb{H}_u(\zeta_x)(g(x_{cs})))
= (g(x_{cs}),\mathbb{H}_u(\mathbb{A}_u \zeta_x)(g(x_{cs})))
\nonumber\\
=&~ \Phi(g(x_{cs}),(\mathbb{A}_u \zeta_x)(g(x_{cs})))=\Phi(g(x_{cs}),A_u(g^{-1}(\pi_c g(x_{cs})))x_u)
\nonumber\\
=&~\Phi(g(x_{cs}),A_u(x_c)x_u),\quad \forall x\in U,
\end{align}
where $\pi_c g(x_{cs})=g(x_c)$, as known below \eqref{B-conju}.
This proves \eqref{PSI-F}.

Next, we show that the definition of $\Phi(x)$ does not depend on $\zeta_x\in \mathcal{X}_\infty$, i.e., if $\zeta_1,\zeta_2\in \mathcal{X}_\infty$ satisfy $\zeta_1(x_{cs})=\zeta_2(x_{cs})=x_u$, then
\begin{align}\label{Hu=Hu}
\mathbb{H}_u(\zeta_1)(x_{cs})=\mathbb{H}_u(\zeta_2)(x_{cs}).
\end{align}
For this purpose, we recall from the proof of \cite[Lemma 6.2]{ZLZ-TAMS} that $\mathbb{G}_u$ can be linearized in $(p-k)$ steps, each of which linearizes the
$\ell$-th component of $\mathbb{G}_u$ (according to the   spectral decomposition) by the following transformations
$$
\text{$\phi_{\ell}$ \, and \, $\varphi_\ell:=\lim_{n\to \infty} \mathbb{A}_\ell^{n} \pi_\ell \mathbb{G}_u^{-n}$,   \, $\forall \ell=k+1,...,p,$}
$$
where $\phi_\ell$ straightens up the weak-unstable manifold $\mathcal{M}_{\ell}$ of $\mathbb{G}_u$ tangent to the subspace corresponding to the spectrum $\bigcup_{i=k+1}^{\ell-1} [\lambda_i-\varsigma,\lambda_i+\varsigma]$, and $\mathbb{A}_\ell$
and $\pi_\ell \mathbb{G}_\ell$ are the $\ell$-th component of $\mathbb{A}_u$ and $\mathbb{G}_u$ (with some modifications in each step), respectively.

Then, we can see that the restriction of $\mathcal{M}_{\ell}$ onto the unstable leaf ${\mathcal W}_u(x_{cs})$ (i.e., the fiber with the base point $x_{cs})$ is actually the leaf ${\mathcal W}_\ell(x_{cs})$ of a weak-unstable foliation of $\tilde{F}$ passing through $x_{cs}$, i.e.,
\begin{align}\label{M-W}
(\mathcal{M}_{\ell})_{x_{cs}}:=\mathcal{M}_{\ell}|_{{\mathcal W}_u(x_{cs})}={\mathcal W}_\ell(x_{cs}).
\end{align}
Indeed, \eqref{barF-FP}, \eqref{Fuu} and the fact $\zeta_x(x_{cs}) = x_u$ give
\begin{align*}
\mathbb{G}_{u}^{n}(\zeta_x)(g^n(x_{cs}))
=
G_{g^{n-1}(x_{cs})}\circ G_{g^{n-2}(x_{cs})}\circ \cdots \circ G_{x_{cs}}(x_u)
=\tilde{F}^n(x)-g^n(x_{cs})
\end{align*}
for all $n\ge 0$, which implies that
$
\varrho_{\ell}^{-n}\|\tilde{F}^n(x)-g^n(x_{cs})\|\le \varrho_{\ell}^{-n}\|\mathbb{G}_{u}^{n}(\zeta_x)\|_{\infty}
$
with $\varrho_{\ell}\in (\lambda_{\ell}+\varsigma,\lambda_{\ell+1}-\varsigma)$. By a definition for the weak-unstable foliation, which is similar to
\eqref{SF-1}, we know that
${\mathcal W}_\ell(x_{cs})\subset (\mathcal{M}_{\ell})_{x_{cs}}$, and therefore \eqref{M-W} holds because ${\mathcal W}_\ell(x_{cs})$ is the graph of a mapping which is unique. Hence, the restriction of $\phi_\ell$ onto ${\mathcal W}_u(x_{cs})$ is only determined by ${\mathcal W}_\ell(x_{cs})$, which implies that
$$
\phi_\ell(\zeta_1)(x_{cs})=\phi_\ell(\zeta_2)(x_{cs})
$$
for any $\zeta_1,\zeta_2\in \mathcal{X}_\infty$ such that $\zeta_1(x_{cs})=\zeta_2(x_{cs})$.

Moreover, for $\varphi_\ell$, we note that
\begin{align*}
  \varphi_\ell(\zeta_x)(x_{cs})= \lim_{n\to \infty}
A_{\ell}(x_c) \cdots A_{\ell}(g^{-n+1}(x_c)) \pi_{\ell}G_{g^{-n+1}(x_{cs})}^{-1} \circ \cdots \circ G_{x_{cs}}^{-1}(x_u).
\end{align*}
It implies that, for any $\zeta_1,\zeta_2\in \mathcal{X}_\infty$ such that $\zeta_1(x_{cs})=\zeta_2(x_{cs})=x_u$, we have
$$
\varphi_\ell(\zeta_1)(x_{cs})=\varphi_\ell(\zeta_2)(x_{cs}).
$$
All the above discussion indicates  that \eqref{Hu=Hu} holds since
$\mathbb{H}_u$ is the composition of $\phi_\ell$ and $\varphi_\ell$.

Now, we define $\Phi^{-1}(x):=(x_{cs},\mathbb{H}_u^{-1}(\zeta_x)(x_{cs}))$ and verify from \eqref{P-H1} that
\begin{align*}
\Phi \circ \Phi^{-1}(x)=&~\Phi(x_{cs},\mathbb{H}_u^{-1}(\zeta_x)(x_{cs}))
=(x_{cs},\mathbb{H}_u \circ \mathbb{H}_u^{-1}(\zeta_x)(x_{cs}))
\\
=&~ (x_{cs},x_u)=x,
\end{align*}
which implies that $\Phi$ is a homeomorphism.

{\em Step 3. Verify the regularity of $\Phi$.}

For a fixed $x=(x_{cs},x_u)\in \mathbb{R}^d$, we define a mapping $i_{x}\in \mathcal{X}_\infty$ by
$$
i_{x}(y_{cs}):=e^{-\|y_{cs}-x_{cs}\|}x_u,\quad \forall y_{cs}\in X_{cs},
$$
which satisfies that $\|i_x\|_\infty=\|x_u\|$.
Define
\begin{align}\label{Phi-uu}
\Phi(x):=(x_{cs},\mathbb{H}_u(i_{x})(x_{cs})),
\end{align}
which is the same as the $\Phi$ defined in \eqref{P-H1} since $i_x(x_{cs})=\zeta_x(x_{cs})=x_u$ by \eqref{Hu=Hu}. Then, by \eqref{HH-u} and \eqref{Phi-uu}, we can check that
\begin{align*}
\frac{\|\Phi(x)-x\|}{\|x-x_{cs}\|^{1+\beta}}
&
= \frac{\|(x_{cs},\mathbb{H}_u(i_{x})(x_{cs}))-(x_{cs},x_u)\|}{\|x_u\|^{1+\beta}}
\\
&= \frac{\|\mathbb{H}_u(i_{x})(x_{cs})-i_x(x_{cs})\|}{\|x_u\|^{1+\beta}}
 \le \frac{\|\mathbb{H}_u(i_{x})-i_{x}\|_\infty}{\|i_{x}\|_\infty^{1+\beta}}\le L
\end{align*}
near the origin for $x\ne x_{cs}$.

Finally, for any  $\tilde{x}_{c}\in X_{c}$, since
$\|x-x_{cs}\|\le \|x-x_c\| \le \|x-\tilde{x}_{c}\|$, we see that
\begin{align*}
\frac{\|\Phi(x)-x\|}{\|x-\tilde{x}_{c}\|^{1+\beta}}
\le   \frac{ L \|x-x_{cs}\|^{1+\beta}}{\|x-\tilde{x}_{c}\|^{1+\beta}}
 \le L,
\end{align*}
which proves \eqref{Dscu-9} for $\Phi$. Similarly, one can prove \eqref{Dscu-9} for $\Phi^{-1}$ and
this completes the proof.
\end{proof}

Remark that,  through \eqref{Auu-g}, the equality \eqref{barF-Ac} obtained in Lemma \ref{co-redu} plays a key role in \eqref{verf-conj}. Otherwise, without \eqref{barF-Ac}, we could only verify that 
$$
\tilde{F}\circ \Phi(x)
=\Phi(g(x_{cs}),A_u(x_{cs})x_u)
$$ 
holds. This would be inconsistent with the Takens’ normal form \eqref{Co-H}, whose linear part only depends on $x_c$.

\appendix
\section{}\label{Appen-A}

In this appendix,   we introduce the Whitney extension theorem   and a lemma for estimating the H\"{o}lder exponent of   functional series.

For $\mathbf{j}:=(j_1,...,j_d)$, $\mathbf{k}:=(k_1,...,k_d)$,
let $\mathbf{j}!:=j_1 ! \cdots j_d !$, $|\mathbf{j}|:=j_1+\cdots +j_d$
and $z^{\mathbf{j}}:=z_1^{j_1}\cdots z_d^{j_d}$, where $d\in\mathbb{N}$,
$j_i,k_i\ge 0$ are integers for all $i=1,...,d$ and $z:=(z_1,...,z_d)\in\mathbb{R}^d$.
For a constant $r\in\mathbb{N}$ and a function $w:\Omega \to \mathbb{R}$, where $\Omega\subset\mathbb{R}^d$ is a closed subset,
suppose that there exists a family of functions $w_{\mathbf{j}}:\Omega\to \mathbb{R}$, $0<|\mathbf{j}|\le r$, such that
\begin{align}\label{W.1} 
 w_{\mathbf{0}}=w, \quad w_{\mathbf{j}}(x)=\sum_{|\mathbf{j}+\mathbf{k}|\le r} \frac{w_{\mathbf{j}+\mathbf{k}}(y)}{\mathbf{k}!}(x-y)^{\mathbf{k}}+R_{\mathbf{j}}(x,y)
\end{align}
and
\begin{align}\label{W.2} 
  |w_{\mathbf{j}}(x) | \le N, \quad    |R_{\mathbf{j}}(x,y)|\le N \|x-y\|^{r+\alpha-\mathbf{j}}, \quad \forall x,y\in \Omega,
\end{align}
with $\alpha\in (0,1]$. Then in view of \cite[Theorem 4]{St-Book}, we have the following lemma.

\begin{lemma}\label{Exten-lem}
Let the functions $w:\Omega\to \mathbb{R}$ and $w_{\mathbf{j}}:\Omega\to \mathbb{R}$, $0\le |\mathbf{j}| \le r$, be given above such that \eqref{W.1} and \eqref{W.2} hold. Then there is a $C^{r,\alpha}$ function $W:\mathbb{R}^d \to \mathbb{R}$ such that
$$
W|_{\Omega}=w, \quad D^{\mathbf{j}}W|_{\Omega}=w_{\mathbf{j}}, \quad 0<|\mathbf{j}|\le r,
$$
where $D^{\mathbf{j}}W:=\partial^{\mathbf{j}}W(x_1,...,x_d)/\partial x_{j_1}\cdots \partial x_{j_d}$.
\end{lemma}

\begin{lemma}\label{Holder-lem}\cite[Lemma 4]{ZZJ-MA}
Let $(Y,\|\cdot\|)$ be a Banach space and $\Omega\subset Y$ be a small neighborhood of the origin $O\in Y$.
Assume that $\Gamma_k:\Omega\to Y$ are $C^{0,\alpha}$ mappings for all integers $k\ge 0$, where $\alpha\in (0,1]$, and that $\tau_1,\tau_2$ are
positive constants such that
$$
\|\Gamma_k(y)\| \le C \tau_1^k, \quad \|\Gamma_k(y)-\Gamma_k(\tilde{y})\| \le C \tau_2^k \|y-\tilde{y}\|^\alpha, \quad \forall y,\tilde{y}\in \Omega.
$$
If there exists a $\rho>0$ such that $\rho\tau_1<1$, then
$$
\sum_{k=0}^\infty \rho^k \|\Gamma_k(y)-\Gamma_k(\tilde{y})\| \le C \|y-\tilde{y}\|^{\bar{\beta}}
$$
near $O$, where
$$
\bar{\beta}:=
\begin{cases}
\alpha, \quad &\rho\tau_2<1,
\\
\alpha-\epsilon, \quad & \rho\tau_2=1,
\\
(\log\tau_1+\log\rho)(\log\tau_1-\log\tau_2)^{-1}\alpha, \quad & \rho\tau_2>1,
\end{cases}
$$
defined for an arbitrarily small given $\epsilon>0$, is a positive constant.
\end{lemma}

%%%%%%%%%%%%%%%%%%%%%%%%%%
\section{}\label{Appen-B}

\begin{proof}[Proof of Lemma {\rm \ref{lm-diagC}}.]
We divide the proof of the lemma into three steps.

{\em Step 1.   Review the classical smooth stable {\rm (}resp. unstable{\rm )} foliation on $\mathbb R^d$.} 
%the stable-center manifold $X_{cs}$ {\rm (}resp. the center-unstable manifold $X_{cu}${\rm )}.}

Let $F$ be a $C^{1,\alpha}$ diffeomorphism such that  \eqref{Df-hold-1} and \eqref{sp-A} hold.  
We only discuss the stable foliation of $F$ on $\mathbb R^d$ 
%defined on $X_{cs}$, 
as 
the discussion for the unstable one is similar, and see that there is a $C^{0}$ mapping $h_s:\mathbb R^d \times X_s \to X_{cu}$ such that the leaves can be given by 
\begin{align*}
\mathcal{W}_s(x)=\{z_s+ h_s(x,z_s):z_s\in X_s\},  \quad \forall x \in \mathbb R^d.
\end{align*}
In fact,   consider the following classical Lyapunov-Perron equation associated with the stable  foliation on $\mathbb R^d$ (cf. \cite{CHT-JDE}): 
\begin{align}\label{LP-cs-MA}
p_n(x,z_s)=&~ A_s^n (z_s-\pi_s x) 
+\sum_{k=0}^{n-1} A_s^{n-k-1}
  \{ \pi_s  f(F^k(x)+p_k(x,z_s))- \pi_s f(F^k(x)) \}
\notag \\
&~ - \sum_{k=n}^\infty A_{cu}^{n-k-1} 
 \{ \pi_{cu}  f(F^k(x)+p_k(x,z_s))- \pi_{cu} f(F^k(x)) \}, \quad \forall n\ge 0.
\end{align}
Using the same arguments as in \cite[Section 2]{CHT-JDE}, 
we can obtain that \eqref{LP-cs-MA} has a unique solution
$(p_n(x,z_s))_{n\ge 0}$ such that $p_n(x,x_s)=0$ and
\begin{align}\label{con-diff-LP}
 \sup_{n\ge 0}\{\gamma_1^{-n}\|p_n(x,z_s)\|\}\le M, \quad 
 \sup_{n\ge 0} \{\gamma_1^{-n} \|\partial_{z_s} p_n(x,z_s)\|\}\le M,
\end{align}
where $M>0$ is a constant, and $\gamma_1 \in(\lambda_s^+,1-\varsigma)$ with small $\varsigma>0$.
Then, by applying \cite[Theorem 3.1(v)]{CHT-JDE},  the $C^0$ stable foliation of $g$  can be  defined by 
\begin{align}\label{def-hs}
h_s(x,z_s):=x_{cu}+\pi_{cu} p_0(x,z_s),    
\end{align}
which is $C^1$ in $z_s$ due to \eqref{con-diff-LP}. 

Next, we show that 
\begin{align}\label{ps0-sm-hold}
\| \partial_{z_s} p_0(x_c,0)-id_s\| \le \delta_p,   \quad 
\| \partial_{z_s} p_0(x_c,0) - \partial_{z_s} p_0(\tilde x_c,0) \| \le    \delta_p \|x_c-\tilde x_c\|^\alpha
\end{align}
for all $x_c,\tilde x_c\in X_c$, where $id_s$ denotes the identity mapping in $X_s$ and $\delta_p>0$ is sufficiently small.  
Indeed, taking the partial derivative on both sides of \eqref{LP-cs-MA}  with respect to $z_s$, we obtain
\begin{align}\label{s-pnp0}
\begin{split}
\partial_{z_s}p_n(x,z_s)
&= A_s^n + \sum_{k=0}^{n-1} A_s^{n-k-1} D(\pi_s f)(F^k(x)+p_k(x,z_s)) \partial_{z_s} p_k(x,z_s)
\\
&\quad - \sum_{k=n}^\infty A_{cu}^{n-k-1} D(\pi_{cu} f)(F^k(x)+p_k(x,z_s)) \partial_{z_s} p_k(x,z_s), \quad \forall n\ge 0.
%\\
%\partial_{z_s}p_0(x_{cs},z_s)
%&= id_s -  \sum_{k=0}^\infty A_c^{-k-1} D(\pi_c f)(g^k(x_{cs})+p_k(x_{cs},z_s)) \partial_{z_s} p_k(x_{cs},z_s), \quad   n=0,
\end{split}    
\end{align}
Then, we see from \eqref{s-pnp0} and the fact $p_k(x_{cu},0)=0$  that 
\begin{align}\label{p0-00}
\begin{split}
\partial_{z_s}p_n(x_c,0) &= A_s^n + \sum_{k=0}^{n-1} A_s^{n-k-1} D(\pi_s f)(F^k(x_{c}) ) \partial_{z_s} p_k(x_{c},0)
\\
&\quad - \sum_{k=n}^\infty A_{cu}^{n-k-1} D(\pi_{cu} f)(F^k(x_{c}) ) \partial_{z_s} p_k(x_{c},0), \quad \forall n\ge 0,
\end{split}
\end{align}
where the term $\sum_{k=0}^{n-1}$ vanishes when $n=0$. 
It follows from \eqref{Df-hold-1}, \eqref{g-center} and \eqref{con-diff-LP} that 
\begin{align*} 
 \|\partial_{z_s} p_0(x_{c},0)-id_s\| \le &~  \sum_{k=0}^\infty K (1-\varsigma)^{-k-1} \delta_f   \| \partial_{z_s} p_k(x_{c},0)\| 
\notag \\
\le&~   \delta_f KM/(1-\varsigma)  \sum_{k=0}^\infty  (\gamma_1/(1-\varsigma))^k \le \delta_p,   
\end{align*}
where $\delta_p>0$ is sufficiently small because $\delta_f$ is small enough. 
It proves the first inequality of \eqref{ps0-sm-hold}. Note that \eqref{g-center} is invoked in advance here since $A(x_c)$ is a small perturbation of $A$, and thus we may reasonably assume that $A$ also satisfies \eqref{ED}-\eqref{g-center} in this proof.
 
Furthermore, for any $x_c,\tilde x_c\in X_c$,   we see
 from \eqref{p0-00} that   for all $n\ge 0$, 
\begin{align}\label{leaf-difLp}
&\partial_{z_s} p_n(x_{c},0)-\partial_{z_s} p_n(\tilde x_{c},0)   
\notag \\
&= \sum_{k=0}^{n-1} A_s^{n-k-1} \Big\{  D(\pi_s f)(F^k(x_{c}))\partial_{z_s}p_k(x_{c},0) - D(\pi_s f)(F^k(\tilde x_{c}) )\partial_{z_s} p_k(\tilde x_{c},0)\Big\}
\notag \\
& \quad -  \sum_{k=n}^\infty A_{cu}^{n-k-1} \Big\{  D(\pi_{cu} f)(F^k(x_{c}))\partial_{z_s}p_k(x_{c},0) - D(\pi_{cu} f)(F^k(\tilde x_{c}) )\partial_{z_s} p_k(\tilde x_{c},0)\Big\}.
\end{align}
Choosing a $\gamma_2\in (\gamma_1,1-\varsigma)$ and by \eqref{con-diff-LP}, we have
$$
\sup_{n\ge 0}\{\gamma_2^{-n} \|\partial_{z_s} p_n(x_c,0)-\partial_{z_s} p_n(\tilde x_c,0)\|\}<\infty,
$$
and note that
$$
\|F^k(x_c)-F^k(\tilde x_c)\|\le \sup_{\xi\in X_c}\|DF^k(\xi)\|\,\|x_c-\tilde x_c\|\le (\lambda_u^+ +\delta_f)^k\|x_c-\tilde x_c\|,\quad \forall k\ge 0.
$$
Therefore,   \eqref{leaf-difLp}  together with \eqref{Df-hold-1}, \eqref{ED}-\eqref{g-center} and \eqref{con-diff-LP} yields 
\begin{align}\label{holder-psn}
& \gamma_2^{-n}     \|\partial_{z_s} p_n(x_c,0)-\partial_{z_s} p_n(\tilde x_c,0)\|
%%%%%%%%%%%%%%%%%%%%%%%%%%%%%%%%%%%%%%%%%%%%%%%%%%%%%%%%
\notag \\
&\le  \gamma_2^{-n} \Bigg\{ \sum_{k=0}^{n-1} \|A_s^{n-k-1} \{ D(\pi_s f)(F^k(x_{c}) )-D(\pi_s f)(F^k(\tilde x_{c}) )   \} \partial_{z_s} p_k(x_{c},0)  \|
\notag \\
&\quad + \sum_{k=0}^{n-1} \|A_s^{n-k-1} D(\pi_s f)(F^k(\tilde x_{c}) ) \{ \partial_{z_s} p_k(x_{c},0) -\partial_{z_s} p_k(\tilde x_{c},0) \}\|
\notag \\
&\quad + \sum_{k=n}^\infty \|A_{cu}^{n-k-1} \{ D(\pi_{cu} f)(F^k(x_{c}) )-D(\pi_{cu} f)(F^k(\tilde x_{c}) )   \} \partial_{z_s} p_k(x_{c},0)  \|
\notag \\
&\quad + \sum_{k=n}^\infty \|A_{cu}^{n-k-1} D(\pi_{cu} f)(F^k(\tilde x_{c}) ) \{ \partial_{z_s} p_k(x_{c},0) -\partial_{z_s} p_k(\tilde x_{c},0) \}\| \Bigg\}
%%%%%%%%%%%%%%%%%%%%%%%%%%%%%%%%%%%%%%%%%%%%%%%%%%%%%%%%
\notag \\
&\le \gamma_2^{-1} \Bigg\{ \sum_{k=0}^{n-1} K \left(\frac{\lambda_s^+}{\gamma_2}\right)^{n-k-1} \delta_f \|F^k(x_c)-F^k(\tilde x_c)\|^\alpha (\gamma_1/\gamma_2)^k\gamma_1^{-k} \|\partial_{z_s} p_k(x_c,0)\| 
\notag \\
&\quad +   \sum_{k=0}^{n-1} K \left(\frac{\lambda_s^+}{\gamma_2}\right)^{n-k-1} \delta_f \gamma_2^{-k} \|\partial_{z_s} p_k(x_{c},0) -\partial_{z_s} p_k(\tilde x_{c},0)\| 
\notag \\
&\quad + \sum_{k=n}^\infty K \left(\frac{1-\varsigma}{\gamma_2}\right)^{n-k-1} \delta_f \|F^k(x_c)-F^k(\tilde x_c)\|^\alpha (\gamma_1/\gamma_2)^k\gamma_1^{-k} \|\partial_{z_s} p_k(x_c,0)\| 
\notag \\
&\quad +  \sum_{k=n}^\infty K \left(\frac{1-\varsigma}{\gamma_2}\right)^{n-k-1} \delta_f \gamma_2^{-k} \|\partial_{z_s} p_k(x_{c},0) -\partial_{z_s} p_k(\tilde x_{c},0)\| \Bigg\}
%%%%%%%%%%%%%%%%%%%%%%%%%%%%%%%%%%%%%%%%%%%%%%%%%%%%%%%%
\notag \\
&\le C\delta_f \sum_{k=0}^\infty \left(\frac{(\lambda_u^+ +\delta_f)^\alpha \gamma_1}{\gamma_2}\right)^k \|x_c-\tilde x_c\|^\alpha 
+ C\delta_f \sup_{n\ge 0}
\{\gamma_2^{-n}  \|\partial_{z_s} p_n(x_c,0)-\partial_{z_s} p_n(\tilde x_c,0)\|\},
\end{align}
where $C>0$ is a constant and $(\lambda_u^+ +\delta_f)^\alpha \gamma_1<\gamma_2$ since $\gamma_1<\gamma_2$ and 
$(\lambda_u^+ +\delta_f)^\alpha$ is close to $1$ as $\alpha>0$ is small.
Therefore, \eqref{holder-psn} gives
\begin{align*}
 \sup_{n\ge 0}\{\gamma_2^{-n} \|\partial_{z_s} p_n(x_c,0)-\partial_{z_s} p_n(\tilde x_c,0)\|\}\le  \delta_p \|x_c-\tilde x_c\|^\alpha,   \quad \forall n\ge 0, 
\end{align*}
which implies that 
\begin{align*}%\label{holder-ps0LZ} 
\|\partial_{z_s} p_0(x_{c},0)-\partial_{z_s} p_0(\tilde x_{c},0)\|
\le  \delta_p   \|x_c-\tilde x_c\|^\alpha.
\end{align*}
This proves the second inequality of \eqref{ps0-sm-hold}.

{\em Step 2.   Study the tangent spaces of leaves at $x_c\in X_c$.  }

Let $E_s(x_{c}):=\{ y_s +\partial_{z_s} h_s(x_{c},0)y_s\in \mathbb R^d: y_s\in X_s\}$
be the tangent space of the stable foliation $W_s(x)$ at $x_c\in X_c$. 
From \eqref{LP-cs-MA} and \eqref{def-hs},  we see that
$$
id_s+\partial_{z_s} h_s(x_c,0)=\partial_{z_s}(\pi_s p_0)(x_c,0)+ \partial_{z_s} (\pi_{cu} p_0)(x_c,0)=\partial_{z_s} p_0(x_c,0),
$$
which together with \eqref{ps0-sm-hold}  implies that  $E_s(x_c)$ is uniformly close to $X_s$.
On the other hand, it is well known that the tangent spaces of leaves of the stable foliation is the fiber of the stable distribution (see e.g. \cite{Pesin}), which implies the invariance of $E_s(x_c)$ under the action of $DF(x_c)$. Similarly, we can define the fiber of the unstable distribution $E_u(x_c)$, which is uniformly close to $X_u$ and is invariant under $DF(x_c)$.

Then, there exists a transition matrix $P(x_c) \in Gl(d, \mathbb R)$, mapping $E_s(x_c)$ and $E_u(x_c)$ onto $X_s$ and $X_u$, respectively, such that $P(0)=id$, 
\begin{align}\label{ps-holder}
\|P(x_c)-id\|   \le \delta_E, 
\quad \|P(x_c)-P(\tilde x_c)\| \le \delta_E \|x_c-\tilde{x}_c\|^\alpha,  \quad \forall x_c,\tilde x_c\in X_c,   
\end{align}
where $\delta_E>0$ is sufficiently small. 
Moreover, we see that $X_s$ and $X_u$ are invariant under $P(F(x_c))DF(x_c)P(x_c)^{-1}$, which implies that 
\begin{align}\label{diag-pcs}
P(F(x_c))DF(x_c)P(x_c)^{-1}=\left(
\begin{array}{ccc}
*& 0& 0
\\
*& *&0
\\
*& 0&*
\end{array}
\right), 
\end{align}
where the three rows of the matrix correspond to the subspaces $X_c$, $X_s$ and $X_u$ respectively, and $*$ denote elements that may be non-zero.

{\em  Step 3.   Extend $P(x_c)$ via the Whitney's extension theorem. }

Similarly to Step 3 in Section \ref{Sec-6}, we can apply the Whitney's extension theorem (see Lemma \ref{Exten-lem} in Appendix \ref{Appen-A}) to obtain a $C^{1,\alpha}$ diffeomorphism $\Upsilon:\mathbb R^d \to \mathbb{R}^d$ such that 
\begin{align}\label{D-upsi}
\Upsilon(x_c)=x_c, \quad D\Upsilon(x_{c})=P(x_c),\quad \|D\Upsilon(x)-id\|\le \delta_\Upsilon  
\end{align}
for all $x_c\in X_c$ and $x\in \mathbb{R}^d$, where $\delta_\Upsilon>0$ is a small constant.

Finally, we only need to give some details to show the third inequality of \eqref{D-upsi}. More precisely, 
we see from \eqref{ps-holder} that the $C^{1,\alpha}$ norm of $\Upsilon-id$ on $X_c$ is small.
Then, by applying \cite[p.177, Theorem 4]{St-Book} (or \cite[Theorem 2]{Feff-RMI}), the extension of $\Upsilon-id$ onto $\mathbb R^d$ satisfies that  
$$
\|\Upsilon-id\|_{C^{1,\alpha}(\mathbb R^d)} \le C\delta_E,
$$
where $C>0$ is a constant and $\|\cdot\|_{C^{1,\alpha}(\mathbb R^d)}$ denotes the $C^{1,\alpha}$ norm in $\mathbb R^d$. 
Therefore, $\Upsilon$ is a $C^{1,\alpha}$ diffeomorphism defined on $\mathbb R^d$ such that 
\begin{align}\label{Up-small}
\|\Upsilon(x)-x\|_{C^{1,\alpha}(\mathbb R^d)} \le C\delta_E,\quad
\quad\|\Upsilon^{-1}(x)-x\|_{C^{1,\alpha}(\mathbb R^d)} \le C\delta_E,
\end{align}
by the Inverse Mapping Theorem. Hence, the new diffeomorphism $\Upsilon\circ F \circ \Upsilon^{-1}$ satisfies that 
\begin{align*}
D(\Upsilon\circ F \circ \Upsilon^{-1})(x_c)
&=D\Upsilon(F \circ \Upsilon^{-1}(x_c))DF (\Upsilon^{-1}(x_c))D\Upsilon^{-1}(x_c)
\\
&=D\Upsilon(F(x_c))DF(x_c)D\Upsilon^{-1}(x_c)=P(F(x_c))DF(x_c)P(x_c)^{-1}
\\
&=\left(
\begin{array}{ccc}
*& 0& 0
\\
*& *&0
\\
*& 0&*
\end{array}
\right)
\end{align*}
due to \eqref{D-upsi} and \eqref{diag-pcs}. Here, we still denoted $\Upsilon\circ F \circ \Upsilon^{-1}$ by $F$, and see from the above equalities that \eqref{CC0} holds. Moreover, 
\eqref{Up-small} means that \eqref{Df-hold-1} is unchanged, and $X_c$ is still the center manifold since $\Upsilon(x_c)=x_c$. It follows that \eqref{c-su} is unchanged.
This completes the proof.     
\end{proof}

\section*{Acknowledgement}

We declare that the authors are ranked in alphabetic order of their names and all of them have the same contributions to this paper.

W. M. Zhang is the corresponding author and was supported by NSF-CQ \#CSTB2023NSCQ-JQX0020 and NSFC \#12271070; Y. H. Xia was supported by NSFC \#12571165 and NSF-ZJ \#LZ24A010006; W. N. Zhang was supported by 
NSFC \#12571180,  \#12171336, and 
National Key R\&D Program of China \#2022YFA1005900.

\end{document}